\documentclass[pdflatex,sn-mathphys-num]{sn-jnl}

\usepackage{graphicx}%
\usepackage{multirow}%
\usepackage{amsmath,amssymb}%
\usepackage{mathrsfs}%
\usepackage[title]{appendix}%
\usepackage{xcolor}%
\usepackage{textcomp}%
\usepackage{manyfoot}%
\usepackage{booktabs}%
\usepackage{algorithm}%
\usepackage{algorithmicx}%
\usepackage{algpseudocode}%
\usepackage{listings}%
\usepackage{multirow} 
\usepackage{hhline}  
\usepackage{lscape}
\usepackage{epstopdf}
\usepackage{subcaption}
\usepackage{changepage}%
\usepackage{ragged2e}
\usepackage{float}
\usepackage{color}
\usepackage{xcolor}
\usepackage{pb-diagram}  

\makeatletter
\g@addto@macro\th@plain{\thm@headpunct{}}
\makeatother

\theoremstyle{plain}%
\newtheorem{theorem}{Theorem}[section]
\theoremstyle{definition}%
\newtheorem{remark}{Remark}[section]%

\theoremstyle{definition}%
\newtheorem{definition}{Definition}[section]%
\newtheorem{lemma}{Lemma}[section]%
\newtheorem{assumption}{Assumption}

\begin{document}

\title[Restricted memory Riemannian quasi-Newton bundle method]{A restricted memory quasi-Newton bundle method for nonsmooth optimization on Riemannian manifolds}


\author[1]{\fnm{Chunming} \sur{Tang}}

\author[2]{\fnm{Shajie} \sur{Xing}}

\author*[3]{\fnm{Wen} \sur{Huang}}\email{wen.huang@xmu.edu.cn}
\author[4]{\fnm{Jinbao} \sur{Jian}}

\affil[1]{\orgdiv{School of Mathematics \& Center for Applied Mathematics of Guangxi}, \orgname{Guangxi University}, \orgaddress{\city{Nanning}, \postcode{530004}, \country{P. R. China}}}

\affil[2]{\orgdiv{School of Physical Science and Technology}, \orgname{Guangxi
			University}, \orgaddress{\city{Nanning}, \postcode{530004}, \country{P. R. China}}}

\affil*[3]{\orgdiv{School of Mathematical Sciences}, \orgname{ Xiamen University}, \orgaddress{\city{Xiamen}, \postcode{361005}, \country{P. R. China}}}

\affil[4]{\orgdiv{School of Mathematical Sciences \& Center for Applied Mathematics of Guangxi}, \orgname{Guangxi Minzu University}, \orgaddress{\city{Nanning}, \postcode{530006}, \country{P. R. China}}}


\abstract{In this paper, a restricted memory quasi-Newton bundle method for minimizing a locally Lipschitz continuous function over a Riemannian manifold is proposed. The curvature information of the objective function is approximated by applying a Riemannian version of the quasi-Newton updating formulas. A Riemannian subgradient aggregation technique is proposed and used to significantly reduce the computations in the quadratic programming subproblem when calculating the candidate descent direction. Moreover, a Riemannian line-search procedure is proposed to generate the stepsizes, and the process is finitely terminated under the assumption of a newly proposed Riemannian semismoothness. Global convergence of the proposed method is established: if the serious iteration steps are finite, then the last serious iterate is stationary; otherwise, every accumulation point of the serious iteration sequence is stationary. In addition, a modified algorithm with limited-memory quasi-Newton updates is presented to further reduce the computational cost.
Finally, numerical experiments demonstrate that (i) the quasi-Newton updates accelerate the convergence of the bundle method, (ii) the aggregation technique significantly reduces the computational cost for solving the quadratic programming subproblem, and (iii) the proposed methods outperform the compared state-of-the-art Riemannian optimization methods for locally Lipschitz continuous functions.}




\maketitle

\section{Introduction}\label{sec1}

In this paper, we consider the following Riemannian optimization problem
	\begin{equation*}
		\min f(x), \ \  {\rm s.t.} \ \ x\in \mathcal{M},
	\end{equation*}
	where $\mathcal{M}$ is an $n$-dimensional Riemannian manifold, and $f:\mathcal M\rightarrow\mathbb R$ is a locally Lipschitz continuous function which may be nonconvex and nonsmooth. Such kind of problems frequently arise in
	practical applications, such as quantum computation, computer vision, signal processing, machine learning, and data science; see, e.g.,~\cite{absil2009optimization,Absil2019}.

Many numerical methods have been developed for the case where the objective function $f$ is continuously differentiable on $\mathcal{M}$, including Riemannian versions of the steepest descent methods, Newton's methods, quasi-Newton methods, conjugate gradient methods, and trust-region methods; see, e.g.,  \cite{absil2009optimization,boumal2022intromanifolds,Hu2020,sato2021riemannian,ring2012optimization,huang2015broyden,huang2018riemannian,sakai2020hybrid,sakai2021sufficient} and the references therein.
Compared with a relatively rich literature on smooth cases,
	the research on Riemannian nonsmooth optimization (i.e., $f$ is nonsmooth and possibly nonconvex) seems far from being satisfied.
	In what follows, we summarize some existing related works, which are based on certain Riemannian versions of nonsmooth optimization techniques extended from the classical linear spaces  (see \cite{azagra2005nonsmooth,hosseini2011generalized,Hosseini2013}). Ferreira and Oliveira \cite{ferreira1998subgradient} proposed a Riemannian subgradient method, in which the objective function $f$ is assumed to be convex in the manifold setting, and the geodesics of $\mathcal M$ must be calculated.
	An iteration-complexity bound of this method was derived by Ferreira et al. \cite{ferreira2019iteration} under suitable conditions. 
     For minimizing locally Lipschitz continuous functions on Riemannian manifolds, Grohs and Hosseini \cite{grohs2016varepsilon}  presented an $\varepsilon$-subgradient algorithm, Grohs and Hosseini \cite{grohs2016nonsmooth} discussed nonsmooth trust-region algorithms, Hosseini and Uschmajew \cite{hosseini2017riemannian,hosseini2019gradient} studied the gradient sampling method, and Hosseini et al. \cite{hosseini2018line} proposed a nonsmooth Riemannian line-search algorithm and extended the classical BFGS algorithm to the nonsmooth Riemannian setting. 	
These methods either require the objective function to be geodesically convex, which restricts the range of applications, or rely on the computation of a quadratic programming problem with the size of at least the dimension of the manifold, which can be computationally expensive. Therefore, the existing results from the above literature either only give convergence analyses or numerical experiments restricted to small-dimension problems.
To overcome this difficulty, current research has developed various optimization methods based on structured functions or smoothing techniques, such as Riemannian proximal methods \cite{chen2020proximal,huang2021riemannian,huang2023inexact}, Riemannian proximal Newton methods~\cite{si2024riemannian,HS2024}, semismooth Newton-based augmented Lagrangian
method on matrix manifolds~\cite{zhou2021semi} and Riemannian smoothing gradient-type algorithms \cite{peng2023riemannian,zhang2024riemannian}. Although these methods perform well in solving medium-to-large size optimization problems, their applicability is constrained by specific structural requirements. The objective function must either be expressed as the sum of smooth and nonsmooth components, or it is difficult to construct effective smooth approximations for complex nonsmooth functions.

Another optional approach is the bundle methods, which are regarded as the most efficient and reliable methods for general nonsmooth optimization problems in linear spaces (see, e.g., the monographs \cite{makela2002survey,bagirov2014introduction} and the references therein), which can be viewed as stabilized variants of the cutting-plane method \cite{Cheney1959,Kelley1960}. These methods require solving a quadratic programming subproblem when computing candidate descent direction.
	Depending on the strategies of stabilization, bundle methods can be subdivided into proximal bundle methods \cite{Kiwiel1990,hare2010redistributed,DeOliveira2014,lv2018a,Tang2019,hoseini2019new}, trust-region bundle methods \cite{Schramm1992,Apkarian2008,Liu2019}, level bundle methods \cite{Lemarechal1995,DeOliveira2014OMS,Lan2015,Tang2022}, and quasi-Newton (also called variable metric) bundle methods \cite{lukvsan1999globally,vlvcek2001globally,haarala2007globally, Tang2022PJO} etc. Thanks to the success of bundle methods on traditional nonsmooth optimization problems, it is natural and meaningful to generalize them to solve Riemannian nonsmooth optimization problems. Of course, this generalization is certainly not straightforward. On the one hand, some special tools of Riemannian optimization (such as retraction and vector transport) are necessary to design the algorithms due to the possible nonlinear structure of the manifolds. On the other hand, the convergence analysis requires well combining the mechanism of bundle methods with the theory of Riemannian nonsmooth analysis. To the best of our knowledge, the first work that studies bundle method for Riemannian optimization was done recently by Hoseini Monjezi et al. \cite{hoseini2021proximal}, in which a proximal bundle method is proposed for nonconvex and nonsmooth optimization on Riemannian manifolds. 
 However, the subgradient aggregation technique has not been generalized to the Riemannian setting. Therefore, the size of the quadratic programming problem may get large.

    This paper aims to propose, for the first time, a Riemannian quasi-Newton bundle method applicable to complete Riemannian manifolds with positive injectivity radius.
    Our main contributions are summarized as follows.
	
\begin{itemize}
\item[$\bullet$] Motivated by the studies in \cite{lukvsan1999globally,vlvcek2001globally,haarala2007globally}, we integrate the advantages of Riemannian quasi-Newton methods \cite{huang2018riemannian,huang2015riemannian} and Riemannian bundle method \cite{hoseini2021proximal} to propose a restricted memory quasi-Newton bundle method for minimizing a locally Lipschitz continuous function over a Riemannian manifold. By aggregating the subgradients of null steps, time-consuming quadratic subproblems do not need to be solved. More precisely, we solve after each null step a simple quadratic programming with only three variables.
In contrast, the quadratic subproblems in \cite{hoseini2021proximal} may become difficult to solve when the number of null steps becomes large, since a linear constraint is appended to the subproblem after each null step.
	
\item[$\bullet$] We propose a Riemannian line-search procedure by extending the one used in \cite{lukvsan1999globally,vlvcek2001globally}. Under a newly introduced Riemannian semismoothness assumption, this procedure terminates finitely and produces a trial point such that either a sufficient decrease in the objective function is achieved or the knowledge to the information of $f$ is significantly enriched.
 Global convergence of the proposed method is established:
if the algorithm generates a finite number of serious steps followed by infinitely many null steps, then the last serious iterate is a stationary point of $f$;
otherwise, any cluster point of the sequence of serious iterates is a stationary point of $f$.

\item[$\bullet$] In order to further reduce the computational cost, we present a modified algorithm by incorporating it into the limited-memory BFGS and SR1 updates. Numerical results show that the proposed methods can effectively solve nonsmooth optimization problems on Riemannian manifolds.
\end{itemize}

We note that, when the manifold constraint $\mathcal{M}$ is dropped, i.e., $\mathcal{M} = \mathbb{R}^n$, the proposed Riemannian quasi-Newton bundle method  and its modified variant with limited-memory quasi-Newton updates are not equivalent to any existing Euclidean quasi-Newton bundle methods, but the spirit of the algorithmic framework closely follows from~\cite{lukvsan1999globally,vlvcek2001globally,haarala2007globally}. In particular, the method in  \cite{lukvsan1999globally} can only solve convex optimization problems, and the methods in \cite{vlvcek2001globally,haarala2007globally} 
differ from ours in the update of the matrix $\mathcal{H}_k$ and the generation of stepsizes.

After an earlier version of this paper was posted on arXiv\footnote{\url{https://arxiv.org/abs/2402.18308}}, we learned of two very recent works in~\cite{BHJ2024,HNP2024} proposing different versions of Riemannian bundle methods. In~\cite{BHJ2024}, a Riemannian bundle method is designed for geodesically convex functions without generalizing the subgradient aggregation technique to the Riemannian setting. The theoretical result is stronger than the existing method in the sense that this method converges to a minimizer rather than a stationary point. However, the convexity assumption restricts the application of this algorithm. 
 In~\cite{HNP2024}, a Riemannian bundle method based on a trust-region framework is proposed for nonconvex objective functions. A subgradient aggregation technique for the Riemannian trust-region framework is proposed, and the global convergence is established.
 However, no practical approaches for generating the second-order information are considered therein for the Riemannian bundle method. Moreover, 
 this aggregation technique ensures computational efficiency and numerical stability by controlling the size of the bundle. When handling large-scale problems, a larger size of the bundle is often required to maintain the model accuracy, which leads to a corresponding rise in the dimension of the trust-region subproblem and thus increases the difficulty of solving it. In contrast, the aggregation technique proposed in this paper yields a subproblem that remains three-dimensional regardless of the dimension of the original problem.

This paper is organized as follows. In Section \ref{sec2}, some standard notations, concepts, and conclusions related to the proposed method are described. In Section \ref{sec3} we specify the restricted memory quasi-Newton bundle method and the required line-search procedure. Convergence analysis is established in Section \ref{sec4}, and a modified
algorithm is presented in Section \ref{sec5}. The numerical experiments and some conclusions are given in Section \ref{sec6} and Section \ref{sec7}, respectively.

\section{Preliminaries}\label{sec2}

In this section, we first recall the outline of the traditional quasi-Newton bundle methods discussed in \cite{lukvsan1999globally,vlvcek2001globally,haarala2007globally}, and then review some standard notations, basic tools, and related results for nonsmooth analysis of Riemannian optimization.
	
\subsection{Outline of the quasi-Newton bundle methods in $\mathbb{R}^n$}\label{sec2.1}
	
In this subsection, we consider momentarily the case where $\mathcal M = \mathbb{R}^n$. Bundle methods belong to the class of first-order black-box methods, i.e., for any given point $x\in \mathbb{R}^n$, only the function value $f(x)$ and one arbitrary subgradient $g\in \partial f(x)$ are available, where $\partial f(x)$ represents the Clarke subdifferential \cite[Sec. 2.1]{Clarke1983} of $f$ at $x$.
Letting $k\in \{1,2,\dots\}$ be the iteration index, this class of methods produces two interrelated iteration sequences \cite{Kiwiel1990}: a sequence of trial points $\{y_k\}$ which is used to provide first-order information of the function $f$, and a sequence of basic points $\{x_k\}$ which is expected to converge to a stationary point of $f$. In the following description, we will see that if $f$ achieves a sufficient descent at $y_{k+1}$, then set $x_{k+1}=y_{k+1}$; otherwise set $x_{k+1}=x_k$.
In this sense, the point $x_k$ is often called the {\it stability center} of iteration $k$, since it corresponds to the ``best" estimate of the minimum obtained so far.

The quasi-Newton bundle methods \cite{lukvsan1999globally,vlvcek2001globally,haarala2007globally} calculate the search direction by
$$d_k=-H_k\tilde{g}_k,\ \ {\rm for}\ k\geq 1,$$
where $H_k \in \mathbb{R}^{n\times n}$ is a symmetric positive-definite matrix, which is an approximation to the inverse Hessian matrix of $f$ at $x_k$ if $f$ is smooth around $x_k$,
and $\tilde{g}_k$ is an aggregation (namely, a convex combination) of some past subgradients $g_j\in\partial f(y_j)$, $j\in J_k\subseteq\{1,\dots,k\}$. We call $\tilde{g}_k$
an aggregate subgradient which plays the same role as the gradient of $f$ used in the standard quasi-Newton methods.

A line-search \cite[Sec. 2]{vlvcek2001globally} is then performed to determine two stepsizes $t^k_{\rm L}\geq 0$ and $t^k_{\rm R}>0$, and define the new iterates as
$$x_{k+1}=x_k+t^k_{\rm L}d_k\ \ {\rm and}\ \  y_{k+1}=x_k+t^k_{\rm R}d_k,\ \ {\rm for}\ k\geq 1$$
with $y_1=x_1$. In particular, if the objective function $f$ achieves a sufficient descent at $y_{k+1}$, i.e.,
$$
f(y_{k+1})\leq f(x_k)-\theta_{\rm L}t^k_{\rm R}w_k
$$
holds with suitably large $t^k_{\rm R}$, where $\theta_{\rm L}\in (0,1/2)$ is a fixed line-search parameter, and $w_k>0$ is the desirable amount of descent for $f$ at $x_k$,
then set $t^k_{\rm L}=t^k_{\rm R}>0$, $x_{k+1}=y_{k+1}$, and declare a \textit{serious step} (or a descent step).
In this case, the stability center is updated to $y_{k+1}$.

Otherwise, a \textit{null step} is declared if $t^k_{\rm R}>t^k_{\rm L}=0$,
$x_{k+1}=x_k$ and
\begin{equation}\label{null-step-cond}
    -\alpha_{k+1} + g_{k+1}^{\rm T}d_k  \geq -\theta_{\rm R}w_k,
\end{equation}
where $\theta_{\rm R}\in (\theta_{\rm L},1/2)$ is fixed, $g_{k+1}\in\partial f(y_{k+1})$, and
$\alpha_{k+1}$ is the subgradient local measure (see \cite[Sec. 3.2]{Kiwiel1985}) defined as
\begin{equation}\label{linear-error}
\alpha_{k+1}=\max\{\vert f(x_k)-[f(y_{k+1})+  g_{k+1}^{\rm T}(x_k-y_{k+1})]\vert,\ \gamma\|x_k-y_{k+1}\|^\nu\},
\end{equation}
with $\nu\geq 1$ and $\gamma$ being a positive constant.
Note that, if $f$ is convex, then one can set $\gamma=0$, and therefore $\alpha_{k+1}$ reduces to the classical linearization error of $f$ linearized at $y_{k+1}$ and evaluated at $x_k$. The condition~(\ref{null-step-cond}) guarantees that after a null step the new trial point $y_{k+1}$ can provide enough information about $f$ in order to ensure convergence.

The next goal is to update the matrix $H_k$ and the aggregate subgradient $\tilde{g}_k$.
Note that the subgradient aggregation only happens in null steps, aiming to accumulate
enough information about the function until a serious step happens.
Once a serious step is taken, the algorithm will make a bundle reset and restart from the latest stability center. In what follows, the updates will be considered in two cases \cite{lukvsan1999globally,vlvcek2001globally}.

(i) {\it Update after a null step}.
Denote by $m$ the lowest index $j$ such that $x_j=x_k$, namely, the index of
iteration after the last serious step, which also implies that $y_m=x_m=x_{m+1}=\ldots=x_k$.
The new aggregate subgradient $\tilde{g}_{k+1}$ is generated by making a convex combination of three subgradients:
the basic subgradient $g_m\in \partial f(x_k)$,
the current aggregate subgradient $\tilde{g}_k$, and
the new subgradient $g_{k+1}\in \partial f(y_{k+1})$, i.e.,
\begin{equation}\label{Euclidean-agg1}     \tilde{g}_{k+1}=\lambda_{k}^{(1)}g_m+\lambda_{k}^{(2)}g_{k+1}+\lambda_{k}^{(3)}\tilde{g}_{k},
\end{equation}
where the coefficient vector $\lambda_k = (\lambda_{k}^{(1)}, \lambda_{k}^{(2)}, \lambda_{k}^{(3)})^{\rm T} \in \mathbb{R}^3$ is determined
by solving the quadratic subproblem
\begin{equation}\label{Eulidean-subprob}
\begin{array}{cll}
\lambda_k \in \arg\min_{\lambda} &\varphi(\lambda) :=& (\lambda^{(1)} g_m+\lambda^{(2)} g_{k+1}+\lambda^{(3)} \tilde{g}_k)^{\rm T}H_k(\lambda^{(1)} g_m+\lambda^{(2)} g_{k+1}+\lambda^{(3)} \tilde{g}_k) \\
&&+2(\lambda^{(2)}\alpha_{k+1}+\lambda^{(3)}\tilde{\alpha}_k)\\
~~~~~~~~~~~~{\rm s.t.} & &\sum\limits_{j=1}^{3}\lambda^{(j)}=1,\ \ \lambda^{(j)}\geq 0,\ \ j\in \{1,2,3\},
\end{array}
\end{equation}
in which $\lambda=(\lambda^{(1)},\lambda^{(2)},\lambda^{(3)})^{\rm T}$, and $\tilde{\alpha}_{k}$ and $\alpha_{k+1}$ are the
aggregate subgradient local measure and the new subgradient local measure, respectively.  Problem (\ref{Eulidean-subprob}) is convex because the feasible region{\color{magenta}, the} simplex $\{\lambda\in\mathbb{R}_+^{3}\ \vert\ \sum_{j=1}^3\lambda^{(j)}=1\}${\color{magenta},} is a convex set, and the function $\varphi$ is convex. Furthermore, if $g_m$, $g_{k+1}$, and $\tilde{g}_k$ are linearly independent, then $\varphi$ is strictly convex, and in this case, problem (\ref{Eulidean-subprob}) has a unique minimizer.
In fact, in (\ref{Euclidean-agg1}), the aggregation of the two subgradients
$g_{k+1}$ and $\tilde{g}_{k}$ (namely, set $\lambda_{k}^{(1)}=0$) are sufficient to ensure global convergence. The purpose of keeping the basic subgradient $g_m$ is to improve the numerical robustness of the algorithm, which is reasonable since $x_k$ is the best point found so far.

Finally, the new matrix $H_{k+1}$ is generated by the SR1 quasi-Newton formula; see~\cite[Sec. 2]{vlvcek2001globally}.

(ii) {\it Update after a serious step}. In this case, the update is simple, since the algorithm
restarts from the new stability center $x_{k+1}$, i.e.,  set
$\tilde{g}_{k+1}=g_{k+1}$ and $m=k+1$, and update $H_{k}$ by the BFGS quasi-Newton formula \cite[Sec. 2]{vlvcek2001globally}.

\subsection{Notations and some basic results for Riemannian nonsmooth optimization}
  In this section, we review some notation, concepts, and known results of differential geometry and Riemannian manifolds, which can be found in, e.g., \cite{lang1999fundamentals,sakai1992riemannian,absil2009optimization,boumal2022intromanifolds,sato2021riemannian}.
   These materials are not new and are provided only for the convenience of readers unfamiliar with Riemannian optimization.

   For ease of reference, we summarize the main notations used throughout this paper in Table~\ref{tab:notations}.
   \begin{table}[h]
  
\centering
\caption{Summary of notations.}
\label{tab:notations}
\begin{tabular}{cl}
\toprule
\textbf{Notation} & \textbf{Description} \\
\midrule
$\mathcal{M}$ & A complete $n$-dimensional Riemannian manifold. \\
$x \in \mathcal{M}$ & A point on the manifold $\mathcal{M}$. \\
$T_{x}\mathcal{M}$ & The tangent space to $\mathcal{M}$ at $x$. \\
$T\mathcal{M}$ & The tangent bundle of $\mathcal{M}$, defined as $\bigcup_{x \in \mathcal{M}} T_x \mathcal{M}$.\\
$\left \langle \cdot,\cdot \right \rangle_{x}$ & The inner product on $T_{x}\mathcal{M}$. \\
$\|\xi\|_x$ & The norm of a tangent vector $\xi \in T_{x}\mathcal{M}$, defined as $\sqrt{\left \langle \xi, \xi \right \rangle_{x}}$. \\
$c:[a,b] \rightarrow \mathcal{M}$ & A smooth curve on $\mathcal{M}$. \\
$\dot{c}(t)$ & The velocity (tangent vector) of the curve $c$ at parameter $t$. \\
$l(c)$ & The length of curve $c$, defined as $l(c) = \int_a^b \|\dot{c}(t)\|_{c(t)} \, \mathrm{d}t$. \\
$\mathrm{dist}(x, y)$ & The distance between $x, y \in \mathcal{M}$, defined as $\inf_{c(a)=x, c(b)=y} l(c)$.\\
$a^\flat$ & The flat of $a \in T_x\mathcal{M}$, defined as the linear map $a^\flat: T_x\mathcal{M} \rightarrow \mathbb{R}$, $v \mapsto \langle a, v \rangle_x$.\\

$\mathcal{A}_x$ & A linear operator on the tangent space $T_x\mathcal{M}$. \\
$\mathcal{A}_x^{*}$ & Adjoint operator defined by $\langle \mathcal{A}_x \eta_x, \xi_x \rangle_x = \langle \eta_x, \mathcal{A}_x^{*} \xi_x \rangle_x$, $\forall \ \eta_x, \xi_x \in T_x\mathcal{M}$.\\

$B(0_x, r)$ & The open ball of radius $r$ in $T_x\mathcal{M}$ about $0_x$, i.e., $\{ \eta_x \in T_x \mathcal{M} : \|\eta_x\|_x < r \}$.\\
$R$ &  A retraction on $\mathcal{M}$, mapping tangent vectors to points on $\mathcal{M}$. \\
$R_x$ & The restriction of retraction $R$ to the tangent space at $x$.\\
$R_x^{-1}$ & The inverse of retraction $R_x$ (locally defined in a neighborhood of $x$). \\
$\mathcal{T}$ & A vector transport on $\mathcal{M}$. \\
$\mathcal{T}^{-1}$ & The inverse of vector transport (if invertible). \\
${\rm id}_{T_x\mathcal{M}}$ & The identity operator on $T_x\mathcal{M}$. \\
${\rm D}R_x(0_x)$ & The derivative of retraction $R_x$ at $0_x$, which is the identity map.\\ 
${\rm D}R_x(\xi_x)$ & The derivative of the retraction $R_x$ at $\xi_x \in T_x\mathcal{M}$. \\
${\rm Inj}_R(\mathcal{M})$ & The injectivity radius of  $\mathcal{M}$ with respect to the retraction $R$. \\
${\rm diag}(A)$ & The diagonal matrix formed from the diagonal entries of matrix $A$. \\
\bottomrule
\end{tabular}
\end{table}

When it is clear from the context, we remove the subscript $x$ and simply write $\left \langle \cdot,\cdot \right \rangle$ and $\|\cdot\|$.
 If $\mathcal{A}_x = \mathcal{A}^{*}_x$, we say $\mathcal{A}_x$ is self-adjoint or symmetric.

The notion of retraction is used to define the update of iterates in Riemannian optimization algorithms. Specifically, a retraction $R$ on a manifold $\mathcal{M}$ is defined as a smooth mapping $R:T \mathcal{M} \rightarrow \mathcal{M}$ such that (i) $R_x(0_x) = x$ and (ii) $\mathrm{D} R_x(0_x): T_x \mathcal{M} \rightarrow T_x \mathcal{M}$ is an identity map, i.e., $\mathrm{D} R_x(0_x) =  {\rm id}_{T_x\mathcal M}$, where $0_x$ denotes the zero vector in $T_x \mathcal{M}$, and $R_x: T_x \mathcal{M} \rightarrow \mathcal{M}$ is the restriction of $R$ to $T_x \mathcal{M}$. An important retraction, called the exponential map, is defined by
	$
	{\rm Exp}:\ T\mathcal M\rightarrow \mathcal M:\ \xi \mapsto {\rm Exp}_x\xi = \gamma(1),
	$
where $\gamma$ is a geodesic (a curve with zero acceleration) satisfying $\gamma(0) = x$ and $\dot{\gamma}(0) = \xi$; see~\cite[Sec. 5.4]{absil2009optimization} for the rigorous definition of geodesic. In a Euclidean space, $R_x(\xi)={\rm Exp}_x\xi = x+\xi$ is the commonly-used update formula of iterates.

The proposed Riemannian bundle method needs to compute linear combinations of tangent vectors at different points. However, in the Riemannian setting, such linear combinations are not defined. The notion of vector transport is therefore introduced. Specifically, a vector transport associated to a retraction $R$ is defined as a continuous function $\mathcal T : T \mathcal{M} \oplus T \mathcal{M} \rightarrow T \mathcal{M}$, $(\eta_x, \xi_x)\mapsto \mathcal{T}_{\eta_x}(\xi_x)$, which satisfies for all $\eta_x,\ \xi_x\in T_x\mathcal M$,
(i) $\mathcal{T}_{\eta_x} :\ T_x\mathcal M\rightarrow T_{R_x(\eta_x)}\mathcal M$ is a linear invertible map, and (ii) $\mathcal{T}_{0_x}(\xi_x) = \xi_x$, where $\oplus$ denotes the Whitney sum. Both $T\mathcal{M}$ and $T\mathcal{M}\oplus T\mathcal{M}$ are endowed with a natural manifold structure, which induces a corresponding topological structure \cite{absil2009optimization,boumal2022intromanifolds}.
This definition differs from \cite[Def. 8.1.1]{absil2009optimization} in that it relaxes the smoothness requirement to continuity, which facilitates practical construction and numerical implementation, especially for handling nonsmooth problems. Additionally, we incorporate {\color{magenta}an} invertibility condition not present in \cite[Def. 8.1.1]{absil2009optimization}. Invertibility is crucial for the algorithm because quasi‑Newton method (such as SR1 and BFGS updates) frequently requires the inverse of the vector transport operator. In practice, invertibility is generally easy to satisfy; 
for example, parallel transport on the sphere naturally possesses this property within the neighborhood $B(0_x,\pi)$. Thus, for solving the maximum of multiple Rayleigh quotients problem \cite{grohs2016varepsilon,hoseini2021proximal} and the sparse vector problem \cite{7547961}, parallel transport ensures the invertibility of operator $\mathcal{T}$ and guarantees valid algorithm iterations.


An isometric vector transport $\mathcal{T}$ additionally preserves the Riemannian metric, i.e.,
	\begin{equation*}\label{vertran1}
		\langle\mathcal{T}_{\eta_x}(\xi_x),\mathcal{T}_{\eta_x}(\zeta_x)\rangle_{R_x(\eta_x)}=\langle\xi_x,\zeta_x\rangle_x.
	\end{equation*}
    We next define a function that plays a key role in the subsequent analysis: 
$$\beta : T\mathcal{M}\rightarrow \mathbb{R},\ \xi_x\mapsto \beta(\xi_x), \text{ where } \beta(\xi_x) =
\begin{cases}
1,& \text{if } \xi_x=0_x,\\
\frac{\|{\rm D}R_x(\xi_x)[\xi_x]\|}{\|\xi_x\|},& \text{otherwise}.
\end{cases}$$
The continuity of this function at $\xi_x=0_x$ is ensured by $\lim_{\xi_x\rightarrow 0_x}\beta({\xi_x})=\lim_{\xi_x\rightarrow 0_x}\frac{\|{\rm D}R_x(\xi_x)[\xi_x]\|}{\|\xi_x\|}=1$; for $\xi\neq 0_x$, the continuity follows directly. For notational brevity, we {\color{magenta}write} $\beta_{\xi_x}:=\beta(\xi_x)$ throughout the following.


A vector transport satisfies the \textit{locking condition}~\cite[Sec. 2]{huang2015broyden} if it further satisfies that
	\begin{equation}\label{vertran2}
		\mathcal{T}_{\xi_x}(\xi_x)=\beta^{-1}_{\xi_x}{\rm D}R_x(\xi_x)[\xi_x],
	\end{equation}
	where
	$
	{\rm D}R_x(\xi_x)[\xi_x]=\frac{\rm d}{{\rm d}t}R_x(t\xi_x)\Big\vert_{t=1}.
	$ Note that the exponential mapping and parallel translation satisfies~\eqref{vertran2} with $\beta^{-1}_{\xi_x} = 1$.

The injectivity radius \cite{hosseini2018line} of $\mathcal M$ with respect to the retraction $R$ is defined as ${\rm Inj}_R(\mathcal{M}):=\inf_{x\in\mathcal{M}}{\color{magenta}{\rm Inj}_R(x)},$ where 
${\color{magenta}{\rm Inj}_R(x)} := \sup\{ r> 0 \ \vert \ R_x: B(0_x, r) \rightarrow R_x(B(0_x,r))\ {\rm is \ {\color{magenta}injective}}\}$ with 
$R_x(B(0_x,r)) = \{R_x(\eta_x)\ \vert \ \eta_x\in B(0_x,r)\}$. Note that if $\mathcal{M}$ is compact, then ${\rm Inj}_R(\mathcal{M})>0$ for any retraction $R$. This follows from \cite[Cor. 10.21]{boumal2022intromanifolds}, which states that the injectivity radius ${\color{magenta}{\rm Inj}_R(x)}$ at each point $x$ is positive for any retraction $R$. Together with the compactness of $\mathcal{M}$ and the Heine–Borel theorem, this yields the desired conclusion (see \cite[Lem. 6.16]{lee2018introduction} for details).
In particular, ${\rm Inj}_R(\mathcal{M})=\infty$ when $\mathcal{M}$ is a Hadamard manifold, and the exponential map is used as a retraction \cite{sakai1992riemannian}.
Therefore, $R^{-1}_x (y)$ is a singleton for all $y \in 
R_x(B(0_x, {\rm Inj}_R(\mathcal{M})))$.
Let $\eta_x$ denote the unique tangent vector $R_x^{-1}(y)$.
We use the following intuitive notations:
	$\mathcal{T}_{x\rightarrow y}(\xi_x) := \mathcal{T}_{\eta_x}(\xi_x),\  \mathcal{T}_{x\leftarrow y}(\xi_y):=\mathcal{T}_{\eta_x}^{-1}(\xi_y).$
In addition, for simplicity of notation we {\color{magenta}write} $\hat{\mathcal{T}}_{x\leftarrow y} = \beta_{\eta_x}\mathcal{T}_{x\leftarrow y}$.

	A function $f:\ \mathcal M \rightarrow \mathbb R$ is said to be Lipschitz continuous near $x\in \mathcal M$ of rank $L$ if there exists a constant $L>0$ such that $\vert f(x)-f(y)\vert\leq L{\rm dist}(x,y)$ holds for all $y$ in an open neighborhood of $x$. A function $f$ is said to be locally Lipschitz continuous on
	$\mathcal M$ if it is Lipschitz continuous near $x$ for all $x\in\mathcal{M}$.
	In this paper, we assume that the objective function $f$ is locally Lipschitz continuous on $\mathcal M$.  Let $X$ be a Hilbert space and $h :\ X \rightarrow \mathbb{R}$ be Lipschitz continuous near a given point $x\in X$.
	The Clarke directional derivative \cite[Sec. 1.2]{Clarke1983} of $h$ at $x$ in the direction $d \in X$ is defined as 
 	$$
	h^\circ(x;d)=\limsup_{y\rightarrow x\atop  t\downarrow 0}\frac{h(y+td)-h(y)}{t}.
 	$$ 	The Clarke subdifferential \cite[Sec. 1.2]{Clarke1983} of $h$ at $x$ is defined by
	$\partial h(x) =\{\xi \in X\ \vert\ \langle\xi, d\rangle\leq h^\circ(x; d),\ \forall d \in X\}.$
    Since the pullback function $\hat{f}_x = f \circ R_x: T_x \mathcal{M} \rightarrow \mathbb{R}$ is defined on a Hilbert space, the Clarke subdifferential of $\hat{f}_x$ is well-defined. The Riemannian Clarke directional derivative \cite[Sec. 2]{hosseini2011generalized} of $f$ at $x$ in the direction $d \in T_x\mathcal{M}$ is defined by $f^\circ(x; d) = \hat{f}_x^{\circ}(0_x; d)$.   The \textit{Riemannian Clarke subdifferential} of $f$ at $x$ is defined by\footnote{Here, we use the same notation $\partial f(x)$ as in linear spaces without leading to any confusion.}
	\begin{equation*}\label{RCSubdiff}
		\partial f(x) = \partial \hat{f}_x(0_x).
	\end{equation*}
	For convenience, the elements of $\partial f(x)$  are also called subgradients of $f$ at $x$.

    \begin{remark}
        The definition of the Riemannian Clarke subdifferential does not depend on the choice of retraction $R$. For any two retractions $R_x$ and $R'_x$, set $f_1=f\circ R_x$ and $f_2=f\circ R'_x$. By definition, it suffices to show $\partial f_1(0_x) = \partial f_2(0_x)$. Define the map $\varphi = (R'_x)^{-1}\circ R_x$. By the properties of retractions, $\varphi$ is a diffeomorphism near $0_x$, and $D\varphi(0_x) = \mathrm{id}_{T_x\mathcal{M}}$.  
Since $f_1=f_2\circ\varphi$, the chain rule for Clarke subdifferentials yields $\partial f_1(0_x)\subseteq \partial f_2(0_x)$; the reverse inclusion follows similarly. For a detailed proof, see \cite{absil2009optimization,boumal2022intromanifolds,hosseini2011generalized}.
    \end{remark}

	
	The following theorem summarizes some fundamental and important properties of the Riemannian Clarke subdifferential, whose proof can be found in \cite[Thm. 2.9]{hosseini2011generalized} and \cite[Thm. 2.3]{hoseini2021proximal}.
	


 \begin{theorem}\label{nonsmo-theo}
		Let $f :\ \mathcal M\rightarrow \mathbb{R}$ be Lipschitz continuous near $x\in\mathcal M$ of rank $L$. Then
		
		(i) $\partial f(x)$ is a nonempty, convex, and compact subset of $T_x\mathcal M$, and $\|\xi\|\leq L$ for all $\xi\in\partial f(x)$;
		
		(ii) $f^\circ(x; d) = \max\{\langle\xi, d\rangle\ \vert\ \xi\in \partial f(x)\}$ for all $d\in T_x\mathcal M$;
		
		(iii) if the sequences $\{x_i\}\subseteq \mathcal M$  and $\{\xi_i\}\subseteq T\mathcal M$ satisfy that $\xi_i \in \partial f(x_i)$ for each $i$, and if $\{x_i\}$ converges to $x$, and $\xi$ is a cluster point of the sequence $\{\mathcal{T}_{x\leftarrow x_i}(\xi_i)\}$, then we have $\xi \in \partial f(x)$;
		
		(iv) $\partial f$ is upper semicontinuous at $x$.
	\end{theorem}
    
    The Clarke subdifferential is adopted in this paper primarily for its theoretical necessity and computational convenience as it not only directly satisfies the key properties required for convergence analysis, specifically Theorem \ref{nonsmo-theo} (i) and (iii), but also has explicit expressions or computable approximations for common nonsmooth optimization problems, making algorithm implementation straightforward.
	
	A necessary condition for reaching a local minimum of $f$ on $\mathcal M$ is given below \cite[Prop. 2.5]{grohs2016varepsilon}.
	\begin{theorem}\label{stationary}
		Suppose that the function $f :\ \mathcal M\rightarrow \mathbb{R}$ is locally Lipschitz continuous, and that $f$ reaches a local minimum at $x$. Then $0_x \in \partial f(x)$.
	\end{theorem}
	
	Based on Theorem \ref{stationary}, it is natural to call $x$ a stationary point of $f$ on $\mathcal M$ if $0_x \in \partial f(x)$. The purpose of our algorithm described in the next section is to find such a point efficiently.

\section{Restricted memory Riemannian quasi-Newton bundle method}\label{sec3}
	
	In this section, we extend the classical quasi-Newton bundle methods to solve nonsmooth optimization problems on Riemannian manifolds. The terminology ``restricted memory" means that at each iteration only a simple subproblem with three variables is solved instead of a time-consuming quadratic subproblem.
	This is achieved by making use of the subgradient aggregation technique.
	The quasi-Newton operator is updated by the Riemannian versions of the SR1 and BFGS formulas, which well balances theoretical properties
	and numerical performance. A Riemannian line-search procedure is executed to determine whether the new step is a serious step or a null step.

\subsection{Algorithm description}\label{sec3.1}
Now we present the details of our algorithm (Algorithm \ref{Algo:RQNBM}), which generalizes the steps described in Section \ref{sec2.1} to the manifold setting.
The proposed algorithm also generates two iteration sequences:
the sequence in $\mathcal{M}$ of trial points $\{y_k\}$ and the sequence in $\mathcal{M}$ of stability centers $\{x_k\}$.
In particular, given an initial iterate $x_1\in \mathcal{M}$, we set $y_1=x_1$, the initial aggregate subgradient $\tilde{g}_1=g_1\in \partial f(y_1)$ and choose an initial symmetric positive-definite operator $\mathcal{H}_1$ on $T_{x_1}\mathcal{M}$. For the $k$th iteration, we define the search direction as
$$d_k=-\mathcal{H}_k\tilde{g}_k,$$
where $\mathcal{H}_k$ is a symmetric positive-definite
linear operator on $T_{x_k}\mathcal{M}$, which approximates the inverse Riemannian Hessian of $f$ at
$x_k$ if $f$ is smooth around $x_k$,
and $\tilde{g}_k$ is an aggregate subgradient which plays the same role as the Riemannian gradient of smooth $f$.

We extend the line-search procedure in Euclidean spaces to the Riemannian setting (see Procedure \ref{Procedure:LS} below), which produces two stepsizes $t_{\rm L}^k\geq 0$ and $t_{\rm R}^k>0$. The next trial point is then defined by $y_{k+1}=R_{x_k}(t^k_{\rm R}d_k)$.
Roughly speaking, if $f$ achieves a sufficient descent at $y_{k+1}$  with a suitably large stepsize (see Remark \ref{remark31}(c.1)), then we have
$t^k_{\rm L}=t^k_{\rm R}>0$ and $x_{k+1}=y_{k+1}$ (serious step); otherwise, we have $t^k_{\rm R}>t^k_{\rm L}=0$ and $x_{k+1}=x_k$ (null step). Note that a null step is also useful, since $y_{k+1}$
may significantly enrich the knowledge of the information of $f$ (see condition (\ref{null-step-cond2})).
After a serious step, a sequence of consecutive null steps is produced until a new serious step is generated. The number of consecutive null steps between two adjacent serious steps may be finite or infinite. For the latter case, we can prove that the last
serious iterate is stationary for $f$ on $\mathcal M$ (see Theorem \ref{Thm:4.2}). Otherwise, if the algorithm generates an infinite number of serious iterates, then every accumulation point of this sequence is stationary (see Theorem \ref{Thm:4.3}).

In order to update $\mathcal{H}_k$, we need to transfer it to the tangent space $T_{x_{k+1}}\mathcal M$ by using a vector transport 
\begin{equation}\label{Hk_wave}
\mathcal{\tilde{H}}_k=\mathcal{T}_{x_k\rightarrow x_{k+1}}\circ\mathcal{H}_k\circ\mathcal{T}_{x_k\leftarrow x_{k+1}}.
\end{equation}
 Then an intermediate operator $\check{\mathcal{H}}_{k+1}$ is generated by applying either the Riemannian SR1 update (see \eqref{SR1update}) or the Riemannian BFGS update (see \eqref{BFGSupdate}), or simply by being set to
 be
 $\mathcal{\tilde{H}}_k$. Specifically, if the current step is a null step, generate $\check{\mathcal{H}}_{k+1}$ using the Riemannian SR1 updating formula. This update 
 guarantees the convergence of the algorithm when an infinite number of consecutive null steps are generated (see Theorem \ref{Thm:4.2}).
After a serious step, the algorithm behaves as if it were
restarted
from the current iterate, and therefore we adopt the more efficient BFGS formula to generate  $\check{\mathcal{H}}_{k+1}$.
The new operator $\mathcal{H}_{k+1}$ is finally obtained by scaling and/or correcting $\check{\mathcal{H}}_{k+1}$.

The update of the aggregate subgradient $\tilde{g}_k$ is divided into two cases.
 After a serious step, since the algorithm will restart from the new stability center $x_{k+1}$, we make an aggregation reset by setting $\tilde{g}_{k+1}=g_{k+1}$. After a null step,
 we need to make a convex combination of the basic subgradient $g_m\in \partial f(x_k)$ (recall that $m$ is the lowest index $j$ satisfying $x_j=x_k$),
 the new subgradient $g_{k+1}\in \partial f(y_{k+1})$, and the current aggregate subgradient $\tilde{g}_k$. However, these three subgradients belong to different linear spaces since $g_m,\ \tilde{g}_k\in T_{x_k}\mathcal M$ but
 $g_{k+1}\in T_{x_{k+1}}\mathcal M$. Thus, we transport $g_{k+1}$ to the tangent space $T_{x_k}\mathcal M$ by a vector transport. The coefficients of the convex combination are obtained by solving a quadratic problem with only three variables. See Remark \ref{remark31} (e) for more details.

\begin{algorithm}
\caption{Riemannian Quasi-Newton Bundle Method (RQNBM)}
\label{Algo:RQNBM}
\begin{algorithmic}[1]
\Require  Initial iterate $x_1\in \mathcal{M}$;  initial symmetric positive-definite operator $\mathcal{H}_1$ on $T_{x_1}\mathcal{M}$; correction parameters $\rho \in(0, 1)$ and $\Gamma\geq 1$;
stepsize control parameters $t_{\rm min}\in (0,1)$ and $t_{\max}\geq 1$;  constant $0<\mu_0<{\rm Inj}_R(\mathcal{M})$; length control $D>0$; tolerance $\varepsilon\geq 0$.

\State Set $y_1=x_1$, and compute $g_1\in \partial f(y_1)$. Initialize 
$\tilde{g}_1=g_1$, $\alpha_1=0$,  $\tilde{\alpha}_1=0$, $w_1=\langle \tilde{g}_1,\tilde{g}_1\rangle$, $m=1$, $i_{\rm C}=i_{\rm U}=n_{\rm C}=0$ and $k=1$. \label{Step:1}

\For {$k=1, 2, \ldots$}     

\If {$w_k\leq\varepsilon$}
 Return $x_k$. \Comment{{\it Stopping criterion}} \label{Step:stoping}
\EndIf

\State Set $d_k=-\mathcal{H}_k\tilde{g}_k$. \Comment{{\it Generate direction}}\label{dk}

\State Select $t^k_{\rm I}\in \left(0, \min\{t_{\max}, \mu_0/\|d_k\|\}\right]$, and call Procedure \ref{Procedure:LS}:  \label{Line-search0}
\State \hskip 1.2cm $(t_{\rm L}^k, t_{\rm R}^k,\delta_{k+1}, \alpha_{k+1})={\rm LS}(x_k, d_k, w_k, t^k_{\rm I}, t_{\rm min})$.  \label{Line-search}\ {\Comment{{\it Line-search}}}

 \State Set $x_{k+1}=R_{x_k}(t^k_{\rm L}d_k)$ and $y_{k+1}=R_{x_k}(t^k_{\rm R}d_k)$. \Comment{{\it Iteration update}} \label{Iteration-update}

 \State Compute $g_{k+1}\in \partial f(y_{k+1})$, and define $u_k = g_{k+1}-\mathcal{T}_{x_k\rightarrow y_{k+1}}(g_m)$, \label{Step:prepare-update1}
 \State $s_k=$ $\mathcal{T}_{x_k\rightarrow x_{k+1}}(t^k_{\rm R}d_k)$, $\tilde{u}_k=\mathcal{T}_{x_{k+1}\leftarrow y_{k+1}}(u_k)$,
$v_k=\mathcal{\tilde{H}}_k\tilde{u}_k-s_k$, where $\mathcal{\tilde{H}}_k$ is defined in~\eqref{Hk_wave}.	\label{Step:prepare-update2}
\If{ $t^k_{\rm L}=0$}  \Comment{{\textit{Update after a null step}}}
\State	Solve subproblem (\ref{mainsubprob}) to obtain $\lambda_k$. \label{Step:solveQP}
\State Compute $\tilde{g}_{k+1}$ and $\tilde{\alpha}_{k+1}$ by (\ref{aggre1}) and (\ref{aggre2}).   \Comment{{\it Subgradient aggregation}} \label{Step:Sub-agg}
\If{$\langle\tilde{g}_{k},v_k\rangle<0$ and either $i_{\rm C}=0$ or (\ref{SR1cond1}) holds} \label{Step:SR1-update-condition}
  \State Set $i_{\rm U}=1$, and calculate $\check{\mathcal{H}}_{k+1}$ by (\ref{SR1update}).  
\Comment{{\it SR1 update}} \label{Step:SR1-update}	
\Else
\State Set $i_{\rm U}=0$ and $\check{\mathcal{H}}_{k+1}=\mathcal{\tilde{H}}_k$.
\EndIf

\Else\ ($t^k_{\rm L}>0$)   \Comment{{\bf{\it Update after a serious step}}}
\State	Set $\tilde{g}_{k+1}=g_{k+1}$, $\tilde{\alpha}_{k+1}=0$, $m=k+1$. 
\Comment{{\it Aggregation reset}} \label{Step:20}
\If{$\langle u_k,s_k\rangle> \rho$}\label{Step:BGFS-update-condition}
    \State Set $i_{\rm U}=1$, and calculate $\check{\mathcal{H}}_{k+1}$ by  (\ref{BFGSupdate}). \Comment{{\it BFGS update}}  \label{Step:BFGS-update}
\Else
   \State  Set $i_{\rm U}=0$ and $\check{\mathcal{H}}_{k+1}=\mathcal{\tilde{H}}_k$.
\EndIf	\label{Step:25}
\EndIf
\If{$\|\check{\mathcal H}_{k+1}\tilde{g}_{k+1}\|>D$}
\State Set $\check{\mathcal{H}}_{k+1} =  (D/\|\check{\mathcal{H}}_{k+1}\tilde{g}_{k+1}\|)\check{\mathcal{H}}_{k+1}$.   \label{Step: scaling} \Comment{{\it Scaling}}
\EndIf
 \State Set
      $\check{w}_{k+1} =\langle\tilde{g}_{k+1},\check{\mathcal{H}}_{k+1}\tilde{g}_{k+1}\rangle+2\tilde{\alpha}_{k+1}.$ \label{Step:30}
     \If{$\check{w}_{k+1}<\rho\|\tilde{g}_{k+1}\|^2$ or $i_{\rm C}=i_{\rm U}=1$}\label{Step:31}
\State Set $w_{k+1}=\check{w}_{k+1}+\rho\|\tilde{g}_{k+1}\|^2$ and $\mathcal{H}_{k+1}=\check{\mathcal{H}}_{k+1}+ \rho {\rm id}_{T_{x_{k+1}}\mathcal{M}}$. \label{Step: correction}
\Comment{{\it Correction}}
\State Set $n_{\rm C}=n_{\rm C}+1$.
    \Else
    \State Set $w_{k+1}=\check{w}_{k+1}$ and $\mathcal{H}_{k+1}=\check{\mathcal{H}}_{k+1}$. \label{Step:no-correction}
    \EndIf \label{Step:36}

   \If{$n_{\rm C}\geq \Gamma$} Set $i_{\rm C}=1$. \label{Step:37}
   \EndIf \label{Step:38}
 \EndFor
\end{algorithmic}
\end{algorithm}


{\floatname{algorithm}{Procedure}
\renewcommand{\thealgorithm}{1}
\begin{algorithm}
\caption{Line-search: $(t_{\rm L}, t_{\rm R}, \delta, \alpha) = {\rm LS}(x, d, w, t_{\rm I}, t_{\rm min})$}
\label{Procedure:LS}
\begin{algorithmic}[1]
\Require Positive parameters $ \theta_{\rm A}, \theta_{\rm L}, \theta_{\rm R}, \theta_{\rm T}$ satisfying $\theta_{\rm T}+\theta_{\rm A}<\theta_{\rm R}<\frac{1}{2}
$ and $\theta_{\rm L}<\theta_{\rm T}$,	$\gamma>0$, $\theta>0$, $\kappa\in(0, \frac{1}{2})$, $\nu\geq 1$.
	
\State Set $t_{\rm A}=0$ and $t=t_{\rm U}=t_{\rm I}$.
\State \label{LS:Step 2} Calculate $f(R_{x}(td))$, $g\in \partial f(R_{x}(td))$ and
	\begin{equation*}\label{lineproc1}	
      \delta=\max \{\vert f(x)-f(R_x(td))+t\langle \hat{\mathcal{T}}_{x\leftarrow R_x(td)}(g),d \rangle\vert,\gamma(t\|d\|)^\nu\}.
	\end{equation*}
\If{$f(R_x(td))\leq f(x)-\theta_{\rm T}tw$} \label{LS:3}
\State Set $t_{\rm A}=t$.
\Else
\State Set $t_{\rm U}=t$.
\EndIf \label{LS:7}
	
\If{$f(R_x(td))\leq f(x)-\theta_{\rm L}tw$ and either $t\geq t_{\rm min}$ or $\delta>\theta_{\rm A}w$} \label{LS:8}
\State Set $t_{\rm R}=t_{\rm L}=t$, $\alpha =0$, and return.
\EndIf\label{LS:10}
	
\If{$-\delta+\langle \hat{\mathcal{T}}_{x\leftarrow R_x(td)} (g), d \rangle \geq -\theta_{\rm R}w$ and $(t-t_{\rm A})\|d\|< \theta$} \label{LS:11}
\State Set $t_{\rm R}=t$, $t_{\rm L}=0$, $\alpha =\delta$, and return.
\EndIf  \label{LS:13}

\State  Choose $t\in [t_{\rm A}+\kappa(t_{\rm U}-t_{\rm A}), t_{\rm U}-\kappa(t_{\rm U}-t_{\rm A})]$, and go to Step \ref{LS:Step 2}.\label{LS:14}
\end{algorithmic}
\end{algorithm}	}

\begin{remark}\label{remark31}
 Some comments on Algorithm \ref{Algo:RQNBM} are given in order.
\begin{itemize}
\item[(a)] In Step \ref{Step:1}, $i_{\rm C}$ denotes the indicator of correction, $i_{\rm U}$ is the indicator of quasi-Newton update, and $n_{\rm C}$ denotes the counter of corrections.

\item[(b)] If the algorithm stops finitely at Step \ref{Step:stoping}, then we obtain that
the current stability center $x_k$ is an (approximate) stationary point (Theorem \ref{Thm:finite-stop}).

\item[(c)] In Steps \ref{Line-search0} and \ref{Line-search}, we propose a Riemannian line-search procedure to generate the stepsizes $t_{\rm L}^k$ and $t_{\rm R}^k$ as well as the Riemannian subgradient local measure $\alpha_{k+1}$, which extends the classical one in Euclidean spaces (see \cite{Kiwiel1985,lukvsan1999globally}) to the manifold setting.
 The choice of the parameter $t^k_{\rm I}\in \left(0, \min\{t_{\max}, \mu_0/\|d_k\|\}\right]$ implies that $t_{\rm R}^k\leq t^k_{\rm I}\leq \mu_0/\|d_k\|$, which ensures that $y_{k+1}=R_{x_k}(t^k_{\rm R}d_k)\in R_{x_k}(B(0_{x_k}, {\rm Inj}_R(\mathcal{M})))$, and therefore $t_{\rm R}^k d_k=R_{x_k}^{-1}(y_{k+1})$ is valid.
This procedure may produce two outcomes as follows.

 (c.1) $t_{\rm R}^k=t_{\rm L}^k>0$, $\alpha_{k+1}=0$ and either
		         $t_{\rm L}^k\geq t_{\min}$ or $\delta_{k+1}>\theta_{\rm A}w_k$, as well as
            \begin{equation*}
	   f(y_{k+1})\leq f(x_k)-\theta_{\rm L}t_{\rm R}^kw_k.
	 \end{equation*}
This case means that $f$ achieves a sufficient descent at $y_{k+1}$ with {\color{magenta}suitably} large $ t^k_{\rm L}$, and therefore declare a serious step. Here, $t \geq t_{\min}$ ensures that sufficient descent is achieved with a sufficiently large stepsize, while $\delta_{k+1} > \theta_A w_k$ requires the subgradient local measure (the linearization error) $\delta_{k+1}$ to exceed $\theta_A w_k$. The latter implies that even if $t \geq t_{\min}$ is not satisfied, as long as the sufficient descent condition holds, the new iterate remains at a certain distance from the current point, and thus the step is still acceptable; see \cite[Sec. 7.2]{Kiwiel1985} for details.

(c.2) $t_{\rm R}^k>t_{\rm L}^k=0$ 
and (recalling the notation  $\hat{\mathcal{T}}_{x\leftarrow y} = \beta_{\eta_x}\mathcal{T}_{x\leftarrow y}$)
\begin{equation}\label{null-step-cond2}
    \ -\alpha_{k+1}+\langle \hat{\mathcal{T}}_{x_{k}\leftarrow y_{k+1}}(g_{k+1}),d_k\rangle \geq -\theta_{\rm R}w_k,
\end{equation}
 where the subgradient locality measure is calculated by
	       \begin{equation}\label{deltak}
		         \alpha_{k+1}=\max \{\vert f(x_k)-f(y_{k+1})+t_{\rm R}^k\langle\hat{\mathcal{T}}_{x_{k}\leftarrow y_{k+1}}(g_{k+1}),d_k\rangle\vert,\ \gamma(\|R_{x_k}^{-1}(y_{k+1})\|)^\nu\}.
		       \end{equation}	
It is clear that relations (\ref{null-step-cond2}) and (\ref{deltak}) generalize \eqref{null-step-cond}
and \eqref{linear-error} in Euclidean spaces, respectively.
In this case, a null step is declared, and condition (\ref{null-step-cond2})
implies that we have obtained enough information about $f$ at $y_{k+1}$ in order to ensure convergence. From the viewpoint that bundle methods are stabilized variants of the cutting-plane method,  condition (\ref{null-step-cond2}) also implies that a high-quality cutting plane is generated at~$y_{k+1}$.

\item[(d)] The notations $u_k$, $s_k$, $\tilde{u}_k$, and $v_k$ defined in Steps \ref{Step:prepare-update1} and \ref{Step:prepare-update2} are used to form the quasi-Newton updating formulas, which are extended from those in Euclidean spaces by using a vector transport.

\item[(e)] In Steps \ref{Step:solveQP} and \ref{Step:Sub-agg}, we aim to update
    the aggregate subgradient $\tilde{g}_k$ and the aggregate subgradient local measure $\tilde{\alpha}_{k}$. This is done by first solving the following quadratic subproblem
    to obtain $\lambda_k = (\lambda_{k}^{(1)}, \lambda_{k}^{(2)}, \lambda_{k}^{(3)})^{\rm T} \in \mathbb{R}^3$:
	\begin{equation}\label{mainsubprob}
		\begin{array}{cll}
			\lambda_k \in \arg\min_{\lambda} &\varphi(\lambda) := \|\lambda^{(1)}\mathcal{W}_kg_m+\lambda^{(2)}\mathcal{W}_k\hat{\mathcal{T}}_{x_k\leftarrow y_{k+1}}(g_{k+1})+\lambda^{(3)}\mathcal{W}_k\tilde{g}_k\|^2\\
   &\ \ \ \ \ \ \ \ \ \ \ +2(\lambda^{(2)}\alpha_{k+1}+\lambda^{(3)}\tilde{\alpha}_k)\\
			~~~~~~~~{\rm s.t.}& \sum\limits_{j=1}^{3}\lambda^{(j)}=1,\ \ \lambda^{(j)}\geq 0,\ \ j\in \{1,2,3\},
		\end{array}
	\end{equation}
where $\lambda = (\lambda^{(1)}, \lambda^{(2)}, \lambda^{(3)})^{\rm T}$, and $\mathcal{W}_k=\mathcal{H}_k^{1/2}$ (which is not computed actually). This problem is convex because the simplex $\{\lambda\in\mathbb{R}_+^3\ \vert \ \sum_{j=1}^3\lambda^{(j)}=1\}$ is convex, and the function $\varphi$ is convex. Furthermore, if $g_m$, $\hat{\mathcal{T}}_{x_k\leftarrow y_{k+1}}(g_{k+1})$, and $\tilde{g}_k$ are linearly independent, then $\varphi$ is strictly convex, and (\ref{mainsubprob}) has a unique minimizer. The updates are given by	
 \begin{eqnarray}
\tilde{g}_{k+1}&=&\lambda_{k}^{(1)}g_m+\lambda_{k}^{(2)}\hat{\mathcal{T}}_{x_{k}\leftarrow y_{k+1}}(g_{k+1})+\lambda_{k}^{(3)}\tilde{g}_{k},\label{aggre1}  \\
		\tilde{\alpha}_{k+1}&=&\lambda_{k}^{(2)}\alpha_{k+1}
		+\lambda_{k}^{(3)}\tilde{\alpha}_{k}.\label{aggre2}
\end{eqnarray}

\item[(f)] 
In Step \ref{Step:SR1-update}, the conditions
\begin{equation}\label{SR1cond1}
 			\rho\|\tilde{g}_{k+1}\|^2\leq \frac{\langle\tilde{g}_{k+1},v_k\rangle^2}{\langle\tilde{u}_k,v_k\rangle}\ \  {\rm and} \ \  \rho n\leq \frac{\|v_k\|^2}{\langle\tilde{u}_k,v_k\rangle}
\end{equation}
serve as one of the judgment criteria in Step~\ref{Step:SR1-update-condition} for the SR1 update. 
The Riemannian SR1 update is of the form 
	\begin{equation}\label{SR1update}
		\check{\mathcal{H}}_{k+1}=\mathcal{\tilde{H}}_k-\frac{v_kv_k^\flat}{\tilde{u}_k^\flat v_k},
	\end{equation}
whose specific structure is utilized to prove convergence; see also the proof of \eqref{wk+1} and \eqref{trace1}.  
Formula (\ref{SR1update}) corresponds to the inverse SR1 update as presented in~\cite[Algo. 1]{huang2015riemannian}.
Note that the SR1 update is executed only in null steps, namely $x_{k+1}=x_k$.
In this case, the vector transports $\mathcal{T}_{x_k\rightarrow x_{k+1}}$ and $\mathcal{T}_{x_k\leftarrow x_{k+1}}$ are both identity maps, and thus the corresponding notations
can be simplified as $\mathcal{\tilde{H}}_k=\mathcal{H}_k$, $\tilde{u}_k=\mathcal{T}_{x_{k}\leftarrow y_{k+1}}(u_k)$, $s_k=t^k_{\rm R}d_k$, and $v_k=\mathcal{{H}}_k\tilde{u}_k-s_k$.

 \item[(g)] 
 In Step \ref{Step:BFGS-update}, after a serious step,
 we adopt the following Riemannian BFGS updating formula \cite[Sec. 3]{qi2010riemannian}
	\begin{equation}\label{BFGSupdate}
	\check{\mathcal{H}}_{k+1}=\mathcal{\tilde{H}}_k-\frac{s_k(\mathcal{\tilde{H}}_k^*u_k)^\flat}
		{u_k^\flat s_k}-\frac{\mathcal{\tilde{H}}_ku_ks_k^\flat}{u_k^\flat s_k}+\frac{a_ks_ks_k^\flat}{(u_k^\flat s_k)^2},
	\end{equation}
	where $\mathcal{\tilde{H}}_k^*$ is the adjoint operator
of $\mathcal{\tilde{H}}_k$, and $a_{k}=\langle u_{k},\mathcal{\tilde{H}}_{k}u_{k}\rangle+\langle u_{k},s_{k}\rangle$. In fact, in the setting of this paper, $\mathcal{\tilde{H}}_k$ is a self-adjoint operator, i.e. $\mathcal{\tilde{H}}_k^*=\mathcal{\tilde{H}}_k$, since $\mathcal{\tilde{H}}_k$ is symmetric, and the vector transport is isometric.
We point out that the BFGS update mainly aims to enhance numerical performance.
Except for preserving positive-definiteness under suitable conditions, the specific structure of \eqref{BFGSupdate} is not used to prove convergence.

  \item[(h)] The conditions $\langle\tilde{g}_{k},v_k\rangle<0$ and
  $\langle u_k,s_k\rangle>\rho$  in Step \ref{Step:SR1-update-condition} and Step \ref{Step:BGFS-update-condition} respectively are used to ensure the positive-definiteness of $\check{\mathcal{H}}_{k+1}$.
  In particular, the condition $\langle\tilde{g}_{k},v_k\rangle<0$ implies  $\langle\tilde{u}_{k},v_k\rangle>0$ (see Lemma~\ref{lemma3}), which guarantees that the relations defined in  (\ref{SR1cond1}) are well-defined.

  \item[(i)] In Step \ref{Step: scaling}, the scaling of $\check{\mathcal{H}}_{k+1}$ ensures
  the boundedness of the sequence $\{\check{\mathcal{H}}_k\tilde{g}_k\}$,  which is required for convergence (see Theorem \ref{Thm:4.3}).
  The correction in Step \ref{Step: correction} is also required to establish global convergence  (see Theorem \ref{Thm:4.2}).

   \item[(j)] When condition $\check{w}_{k+1}<\rho\|\tilde{g}_{k+1}\|^2$ in Step \ref{Step:31} is satisfied for $\Gamma$ times, we set $i_C=1$ by Step
   \ref{Step:37}. 
   If the SR1 update is used in later iterations, then
   the correction in Step \ref{Step: correction} is executed, which could make SR1 update condition untrue, and therefore the frequency of SR1 update is reduced. This may improve the numerical performance.
   \end{itemize}
   \end{remark}

   The following lemma shows the fundamental property of the operators $\check{\mathcal{H}}_k$, and $\mathcal{H}_k$ generated by Algorithm \ref{Algo:RQNBM}.

   \begin{lemma}\label{lemma3}
   The operators $\check{\mathcal{H}}_k$ and $\mathcal{H}_k$ generated by Algorithm \ref{Algo:RQNBM} are all symmetric positive-definite. In addition, if the condition $\langle\tilde{g}_{k},v_k\rangle<0$ in Step \ref{Step:SR1-update-condition} is satisfied, then $\langle\tilde{u}_{k},v_k\rangle>0$.
   \end{lemma}
  \makeatletter
\let\sn@oldproof\proof
\let\sn@endoldproof\endproof
\renewenvironment{proof}[1][\proofname]
  {\begin{sn@oldproof}[#1]\normalsize}  
  {\end{sn@oldproof}}
\makeatother
   \begin{proof}
   For the first part, we prove by induction that the operators $\mathcal{H}_k$ and $\check{\mathcal{H}}_{k+1}$ are symmetric positive-definite for all $k\geq 1$.
   In fact, for $k=1$, it follows from Algorithm \ref{Algo:RQNBM} that $\mathcal{H}_1$ is symmetric positive-definite. Therefore, $\mathcal{H}_1^{-1}$ and $\mathcal{\tilde{H}}_1$ are symmetric positive-definite.
   
   For $\check{\mathcal{H}}_{2}$, if the SR1 update is used, we have $\langle\tilde{g}_1,v_1\rangle<0$ and $\tilde{g}_1\neq 0$.
   From Remark \ref{remark31} (f), it follows that $v_1=\mathcal{{H}}_1\tilde{u}_1-s_1$ and $s_1=t^1_{\rm R}d_1$. Hence
    \begin{equation*}
    \begin{split}
         -\langle\tilde{g}_1,v_1\rangle&=\langle-\tilde{g}_1,\mathcal{{H}}_1\tilde{u}_1-s_1\rangle = \langle{\mathcal{H}}_1^{-1}d_1,\mathcal{{H}}_1\tilde{u}_1-s_1\rangle\\
         &= \frac{1}{t_{\rm R}^1}\langle{\mathcal{H}}_1^{-1}s_1,\mathcal{{H}}_1\tilde{u}_1-s_1\rangle
         = \frac{1}{t_{\rm R}^1}\langle{s_1,\tilde{u}_1-\mathcal{H}}_1^{-1}s_1\rangle,
    \end{split}
    \end{equation*}
   which implies that $\langle{s_1,\tilde{u}_1-\mathcal{H}}_1^{-1}s_1\rangle>0$.
   In addition, from \cite[Algo. 1]{huang2015riemannian}, we know that
        \begin{equation*}\label{SRupdate}	\check{\mathcal{H}}_{k+1}^{-1}=\mathcal{H}_k^{-1}+\frac{(\tilde{u}_k-\mathcal{H}_k^{-1}s_k)(\tilde{u}_k-\mathcal{H}_k^{-1}s_k)^\flat}{\langle s_k,\tilde{u}_k-\mathcal{H}_k^{-1}s_k\rangle}.
       \end{equation*}
   Thus $\check{\mathcal{H}}_{2}$ is symmetric positive-definite. If the BFGS update is used,
   then we have $\langle u_1,s_1\rangle>\rho>0$ which guarantees that $\check{\mathcal{H}}_{2}$ is symmetric positive-definite (see \cite[Sec. 2]{huang2015broyden}). Otherwise, $\check{\mathcal{H}}_{2}=\mathcal{\tilde{H}}_1$, showing that $\check{\mathcal{H}}_{2}$ is symmetric positive-definite.
   
   Now, we assume that $\mathcal{H}_k$ and $\check{\mathcal{H}}_{k+1}$ are symmetric positive-definite. It follows from Steps \ref{Step: correction} and \ref{Step:no-correction} of Algorithm \ref{Algo:RQNBM} that $\mathcal{H}_{k+1}$ is symmetric positive-definite.
   Moreover, by applying the same analysis for $\check{\mathcal{H}}_{k+2}$ above, we obtain that $\check{\mathcal{H}}_{k+2}$ is symmetric positive-definite.

   For the second part. If $\langle\tilde{g}_k,v_k\rangle<0$, by the positive-definiteness of $\mathcal{H}_k$,
    we obtain	$$\langle\tilde{u}_k,d_k\rangle>\langle\tilde{u}_k,d_k\rangle+\langle\tilde{g}_k,v_k\rangle=-\langle\tilde{g}_k,t^k_{\rm R}d_k\rangle= t^k_{\rm R}\langle\tilde{g}_k,\mathcal{H}_k\tilde{g}_k\rangle>0,$$	
   	which implies $\tilde{u}_k\neq 0$, and therefore $\langle\tilde{u}_k,\mathcal{H}_k\tilde{u}_k\rangle>0$.
   	Using the Cauchy inequality, we get
         \begin{equation*}
   			\langle\tilde{u}_k,d_k\rangle^2 =\langle\tilde{g}_k,\mathcal{H}_k\tilde{u}_k\rangle^2
   			\leq  \langle\tilde{g}_k,\mathcal{H}_k\tilde{g}_k\rangle \langle\tilde{u}_k,\mathcal{H}_k\tilde{u}_k\rangle
   			< \frac{1}{t_{\rm R}^k}\langle\tilde{u}_k,\mathcal{H}_k\tilde{u}_k\rangle \langle d_k,\tilde{u}_k\rangle.
   	\end{equation*}
   	Hence,	$$\langle\tilde{u}_k,v_k\rangle=\langle\tilde{u}_k,\mathcal{H}_k\tilde{u}_k\rangle-\langle t_{\rm R}^kd_k,\tilde{u}_k\rangle>0,$$
   which completes the proof.  \qed
   \end{proof}

   \subsection{Well-definedness of the line-search procedure}\label{sec3.2}

   In order to prove the well-definedness of Procedure \ref{Procedure:LS}, we propose a Riemannian version of the semismooth assumption on the objective function $f$, which is an extension of the Euclidean version in
    \cite[Sec. 3.3]{Kiwiel1985} and is weaker than the existing Riemannian version in~\cite[Def. 2.1]{ghahraei2011semismooth}.
   
   \begin{definition}\label{def3.1}
   A locally Lipschitz continuous function  $f:\mathcal{M}\rightarrow \mathbb{R}$ on a Riemannian manifold $\mathcal{M}$ is said to be semismooth at $x\in \mathcal{M}$ if there exists a chart $(U, \varphi)$ at $x$ such that $f\circ\varphi^{-1}:\varphi(U)\rightarrow \mathbb{R}$ is semismooth at $\varphi(x)\in\mathbb{R}^{n}$. That is, for any $\hat{v}\in\mathbb{R}^n$ and sequences $\{\hat{g}_i\}$ and $\{t_i\}\subseteq \mathbb{R}_{+}$ satisfying $\hat{g}_i\in\partial (f\circ\varphi^{-1})(\varphi(x)+t_i\hat{v})$ and $t_i\downarrow 0$, we have
   \begin{equation*}
        \limsup_{i\rightarrow \infty} \hat{g}_i^{\rm T}\hat{v}\geq\liminf_{i\rightarrow \infty}\frac{f\circ\varphi^{-1}(\varphi(x)+t_i\hat{v})-f\circ\varphi^{-1}(\varphi(x))}{t_i}.
   \end{equation*}
    The function $f$ is semismooth on $\mathcal{M}$ if it is semismooth at all $x\in \mathcal{M}$.
   \end{definition}
   
   \begin{remark}
   For Definition \ref{def3.1}, we make the following comments.
   
   \begin{itemize}
   
    \item [(a)]
        If $f$ satisfies the semismooth property defined by \cite[Def. 2.1]{ghahraei2011semismooth}, then from \cite[Prop. 4]{mifflin1977algorithm}, we know that $f\circ\varphi^{-1}$ is weakly upper semismooth (see \cite[Def. 2]{mifflin1977algorithm}). Therefore, it follows that
   $$\liminf_{i\rightarrow \infty} \hat{g}_i^{\rm T}\hat{v}\geq\limsup_{i\rightarrow \infty}\frac{f\circ\varphi^{-1}(\varphi(x)+t_i\hat{v})-f\circ\varphi^{-1}(\varphi(x))}{t_i}.$$
         It is easy to see that if $f\circ\varphi^{-1}$ is weakly upper semismooth then Definition \ref{def3.1} is satisfied. Therefore, a function that satisfies the semismoothness defined in \cite{ghahraei2011semismooth,malmir2022generalized} also satisfies Definition \ref{def3.1}.
         When $\mathcal{M}$ takes $\mathbb{R}$ (i.e., $\mathcal{M}=\mathbb{R}$), an example given in \cite[Sec. 3]{bihain1984optimization} satisfies Definition \ref{def3.1} but does not satisfy  \cite[Def. 2.1]{ghahraei2011semismooth}. In conclusion, the semismoothness defined in this paper is weaker than the one in \cite{ghahraei2011semismooth,malmir2022generalized}.

   
   
       \item [(b)] This definition does not depend on the selection of charts, which can be shown similarly according to \cite[Prop. 2.2]{ghahraei2011semismooth}. In fact, suppose that $f$ is semismooth at $x$, i.e., there exists a chart $(U, \varphi)$ at $x$ such that $f\circ\varphi^{-1}$ is semismooth at $\varphi(x)$. If there is another chart $(V, \psi)$ such that $x\in V$, then $U\cap V$ is a nonempty subset of $\mathcal{M}$. The restriction of the mapping $f\circ \psi$ on the open set $\psi(U\cap V)\subseteq \mathbb{R}^n$ can be expressed as
       $$(f\circ\psi^{-1})\vert_{\psi(U\cap V)}=(f\circ\varphi^{-1})\vert_{\varphi(U\cap V)}\circ(\varphi\circ\psi^{-1})\vert_{\psi(U\cap V)}.$$
      Because of the $C^{\infty}$ structure on $\mathcal{M}$, the map $(\varphi\circ\psi^{-1})\vert_{\psi(U\cap V)}$ is smooth. According to  \cite[Thm. 5]{mifflin1977semismooth}, we conclude that $f\circ\psi^{-1}$ is semismooth at $\psi(x)$.

      \item [(c)] In particular, if $\mathcal{M}=\mathbb{R}^n$,
      take $U=\mathbb{R}^n$ and $\varphi$ being an identity map,
      then Definition \ref{def3.1} coincides with the definition of semismoothness  in \cite[Sec. 3.3]{Kiwiel1985}.
   \end{itemize}
   \end{remark}

   The semismoothness assumption given by Definition \ref{def3.1} leads to an important consequence below, which is applied to prove the well-definedness of Procedure~\ref{Procedure:LS}.


   \begin{lemma}\label{semismooth1}
     Let $f$ be semismooth on $\mathcal{M}$ in the sense of Definition \ref{def3.1}, $R$ be a retraction, and $x \in\mathcal{M}$. There exist a real number $r > 0$ and a neighborhood $U$ of $x$ in $\mathcal{M}$ such that $R_x : B(0_x,r) \to U$ is a diffeomorphism. Moreover, for all $v \in T_x\mathcal{M}$ and sequences $\{g_i\}$ and $\{t_i\} \subseteq \mathbb{R}_+$ such that
$g_i \in\partial f(R_x(t_iv))$, $t_i \downarrow  t^*$, and $t^*v \in B(0_x, r)$, we have
       \begin{align}\label{semism2}
           &\limsup_{i\rightarrow \infty}\langle\mathcal{T}_{R_x(t^*v)\leftarrow R_x(t^iv)} (g_i),{\rm D}R_{x}(t^*v)[v]\rangle\geq\liminf_{i\rightarrow \infty}\frac{f (R_{x}(t_iv))-f(R_{x}(t^*v))}{t_i-t^*}.
       \end{align}
   
   \end{lemma}

   \begin{proof}
        First, for each $x\in\mathcal{M}$, there are real  number $r>0$ and a neighborhood $U$ of $x$ in $\mathcal{M}$ such that $R_x$ is a diffeomorphism on $B(0_x,r)$; see \cite[Prop. 10.20]{boumal2022intromanifolds}.
       In addition,
       since $f$ is semismooth at $R_{x}(t^*v)\in U$, it follows that for
       $\hat{g}_i\in\partial f\circ\varphi^{-1}(\varphi(R_{x}(t^*v))+(t_i-t^*){\rm D}\varphi\circ R_{x}(t^*v)[v])$ and $t_i\downarrow t^*$, we have
       \begin{equation}\label{semism1}
           \begin{aligned}
           & \limsup_{i\rightarrow \infty} \hat{g}_i^{\rm T}{\rm D}\varphi\circ R_{x}(t^*v)[v]\\
          &\geq\liminf_{i\rightarrow \infty}\frac{f\circ\varphi^{-1}(\varphi(R_{x}(t^*v))+(t_i-t^*){\rm D}\varphi\circ R_{x}(t^*v)[v])-f(R_{x}(t^*v))}{t_i-t^*}.
           \end{aligned}
       \end{equation}

       Let $E:\mathbb{R}^n \rightarrow T_x\mathcal{M}$, $(x_1,x_2,\cdots,x_n)\mapsto \sum_{i=1}^{n}x_iE_i$ be a linear bijection, where $\{E_i\}_{i=1}^{n}$ is an orthonormal basis on $T_x\mathcal{M}$.
       From this, we choose the chart as
       $$\varphi:= E^{-1}\circ R_x^{-1}: U\ \rightarrow \mathbb{R}^{n}. $$
      Then
      ${\rm D}\varphi\circ R_{x}(t^*v)[v]={\rm D}E^{-1}(t^*v)[v]=E^{-1}(v)$. Thus, we have
      \begin{align*}
         & \liminf_{i\rightarrow \infty}\frac{f\circ\varphi^{-1}(\varphi(R_{x}(t^*v))+(t_i-t^*){\rm D}\varphi\circ R_{x}(t^*v)[v])-f(R_{x}(t^*v))}{t_i-t^*}\\
      &=\liminf_{i\rightarrow \infty}\frac{f\circ R_x\circ E(E^{-1}(t^*v))+(t_i-t^*)E^{-1}(v))-f(R_{x}(t^*v))}{t_i-t^*}\\
      &=\liminf_{i\rightarrow \infty}\frac{f(R_x(t_iv))-f(R_{x}(t^*v))}{t_i-t^*},
      \end{align*}
      and $\hat{g}_i\in\partial f\circ\varphi^{-1}(\varphi(R_{x}(t_iv)))$.
   
      On the other hand, define $\hat{\xi}_y={\rm D}\varphi(y)[\xi_y]$ and $\hat{\eta}_y={\rm D}\varphi(y)[\eta_y]$ for any $\xi_y,\eta_y \in T_y\mathcal{M}$, then  \begin{equation}\label{matrix1}
          \langle\xi_y,\eta_y\rangle =\hat{\xi}_y^{\rm T}G_{\varphi(y)}\hat{\eta}_y,
      \end{equation}
      where $y=R_x(t^*v)\in U$ and  $G_{\varphi(y)}\in\mathbb{R}^{n\times n}$ represents the metric matrix with respect to the basis $E_y$.
   

   
   
       The left side of inequality (\ref{semism1}) is
      \begin{align*}
           & \limsup_{i\rightarrow \infty} \hat{g}_i^{\rm T}{\rm D}\varphi\circ R_{x}(t^*v)[v]\\
           & =\limsup_{i\rightarrow \infty} (G_{\varphi(R_x(t_iv))}{\rm D}\varphi(R_{x}(t_iv))[g_i])^{\rm T}{\rm D}\varphi\circ R_{x}(t^*v)[v]\\
            &= \limsup_{i\rightarrow \infty} ({\rm D}\varphi(R_{x}(t_iv))[g_i])^{\rm T}G_{\varphi(R_x(t_iv))}{\rm D}\varphi\circ R_{x}(t^*v)[v]\\
            &= \limsup_{i\rightarrow \infty} \langle({\rm D}\varphi(R_{x}(t_iv)))^{-1}{\rm D}\varphi(R_{x}(t_iv))[g_i], ({\rm D}\varphi(R_{x}(t_iv)))^{-1}{\rm D}\varphi\circ R_{x}(t^*v)[v]\rangle \\
            &= \limsup_{i\rightarrow \infty} \langle{g}_i, ({\rm D}\varphi(R_{x}(t_iv)))^{-1}{\rm D}\varphi\circ R_{x}(t^*v)[v]\rangle \\
            &=\limsup_{i\rightarrow \infty}\langle\mathcal{T}_{R_x(t^*v)\leftarrow R_x(t_iv)}{(g_i)},{\rm D}R_{x}(t^*v)[v]\rangle,
      \end{align*}
      where $g_i\in\partial f(R_x(t_iv))$ and $R_x(t^*v)\in U$.  The first equation holds since $\hat{g}_i=G_{\varphi(R_x(t_iv))}{\rm D}\varphi(R_{x}(t_iv))[g_i]$ according to \cite[Prop. 3.1]{yang2014optimality}. The second equation is true because of the symmetry of $G_{\varphi(R_x(t_iv))}$. The third equation
      follows from (\ref{matrix1}).
   
        Summarizing the above discussions, we obtain \eqref{semism2}.\qed
   \end{proof}

   \vspace{3mm}
   The following lemma, which relies on Assumption \ref{assum3} below, shows that the line-search procedure is well-defined. The proof is an extension of the Euclidean version in \cite[Lem. 2.2]{vlvcek2001globally}.
   There are two main differences:
   (a) the application of the result given in Lemma \ref{semismooth1}; and (b) the relation $(t^i-t^*)d_k=(x+t^id_k)-(x+t^*d_k)$ holds immediately in the Euclidean case, while in the Riemannian case, there may not be a known vector representing the relationship between $R_x(t^id_k)$ and $R_x(t^*d_k)$.

     \begin{assumption}\label{assum3}
       The vector transport $\mathcal{T}$ is isometric and satisfies the locking condition~(\ref{vertran2}).
   \end{assumption}
   
   
   
   \begin{lemma}\label{lemma3.2}
   Suppose that $f$ is semismooth and that Assumption \ref{assum3} holds. Then Procedure \ref{Procedure:LS} terminates in finite steps.
   \end{lemma}
   \begin{proof}
   	Supposed by contradiction that the procedure does not terminate.
    For the $i$th iteration, we use the notations $t^i$, $t_{\rm A}^i$, $t_{\rm U}^i$, $g^i$, $\delta^i$ to denote $t$, $t_{\rm A}$, $t_{\rm U}$, $g$, $\delta$ in the proceduce. For any $i$, we have
   	$$t_{\rm A}^i\leq t_{\rm A}^{i+1}\leq t_{\rm U}^{i+1}\leq {t_{\rm U}^i}\ \ {\rm and}\ \  {t_{\rm U}^{i+1}}-{t_{\rm A}^{i+1}}\leq(1-\kappa)(t_{\rm U}^i-{t_{\rm A}^i}).$$
    Since $ t^i\in  [t_{\rm A}^i,{t_{\rm U}^i}]$, there is $t^*\geq 0$ such that ${t_{\rm A}^i}\uparrow t^*$, ${t_{\rm U}^i}\downarrow t^*$, $t^i\rightarrow t^*$ as $i\rightarrow\infty$. 
    
     Define the set
   	$$ S=\{t\geq 0\ \vert\ f(R_{x_k}(t d_k))\leq f(x_k)-\theta_{\rm T} t w_k \}.$$Due to $\{t^i_{\rm A}\}\subseteq S$, ${t^i_{\rm A}}\rightarrow t^*$, and the continuity of $f\circ R_{x_k}$,
    we have
   	\begin{equation}\label{linelem1}
   		f(R_{x_k}(t^*d_k))\leq f(x_k)-\theta_{\rm T} t^*w_k,
   	\end{equation}
   	implying $t^*\in S$.
   Define $\mathcal I=\{i\ \vert\ t^i\notin S\}$. We prove that the set $\mathcal I$ is infinite. Otherwise, there exists $i_{\rm N}$ which is the greatest element in $\mathcal{I}$ such that $t^i\in S$ for any $i>i_{\rm N}$. Therefore we have $t^{i}_{\rm U}= t^{i_{\rm N}}_{\rm U}$ for all $i>i_{\rm N}$. Then $t^*=t^{i_{\rm N}}_{\rm U}\notin S$ which is a contradiction. For all $i\in \mathcal I$, we obtain
   	$$f(R_{x_k}(t^id_k))> f(x_k)-\theta_{\rm T} t^i w_k.\ \ $$
       According to Steps \ref{LS:3}-\ref{LS:7} of Procedure \ref{Procedure:LS}, we have $t^i=t_{\rm U}^{i+1}$. From the facts that $\mathcal{I}$ is an infinite set and $t_{\rm U}^i\downarrow t^*$, we know that $t_i>t^*$.
   	This together with (\ref{linelem1}) shows that
   	$$
   	-\theta_{\rm T} w_k<\frac{f(R_{x_k}(t^i d_k))-f(R_{x_k}(t^*d_k))}{t^i-t^*}
   	$$
   	for all $i\in \mathcal I$. Let $\mathcal{I}_0=\{i\in \mathcal{I}\ \vert\ R_{x_k}(t^i d_k)\in B(R_{x_k}(t^*d_k), {\rm Inj}_R(\mathcal{M}))\}$. Taking the limit inferior in $\mathcal{I}_0$, we obtain
   	\begin{equation}\label{linelem2}
   		\begin{split}
   			-\theta_{\rm T} w_k &\leq \liminf_{i\in \mathcal{I}_0}\frac{f(R_{x_k}(t^i d_k))-f(R_{x_k}(t^*d_k))}{t^i-t^*}\\
   			&\leq \limsup_{i\in \mathcal{I}_0}\langle\mathcal{T}_{R_{x_k}(t^*d_k)\leftarrow R_{x_k}(t^i d_k)}g^i,{\rm D}R_{x_k}(t^*d_k)[d_k]\rangle \\
                 &= \limsup_{i\in \mathcal{I}_0}\beta_{t^*d_k}\langle\mathcal{T}_{R_{x_k}(t^*d_k)\leftarrow R_{x_k}(t^i d_k)}(g^i),\mathcal{T}_{x_k\rightarrow R_{x_k}(t^*d_k)}d_k\rangle \\
   			&= \limsup_{i\in \mathcal{I}_0}\beta_{t^*d_k}\langle\mathcal{T}_{x_k\leftarrow R_{x_k}(t^*d_k)}\mathcal{T}_{R_{x_k}(t^*d_k)\leftarrow R_{x_k}(t^i d_k)}(g^i),d_k\rangle \\
   			&= \limsup_{i\in \mathcal{I}_0}\beta_{t^*d_k}\langle \mathcal{T}_{x_k\leftarrow R_{x_k}(t^i d_k)}(g^i),d_k\rangle,
   		\end{split}
   	\end{equation}
   	where $g^i\in \partial f(R_{x_k}(t^i d_k))$. The second inequality is established by Lemma \ref{semismooth1}, and the first and second equations hold by Assumption \ref{assum3}. The third equation is given by the boundedness of $\{g^i\}$. From Steps \ref{LS:11}-\ref{LS:13} of Procedure \ref{Procedure:LS}, the condition $(t^i-t_{\rm A})\|d_k\|<\theta$ remains true for sufficiently large $i$.
    Next, we will analyze two cases.
   	
   	First, if $t^*=0$, then $t^i\rightarrow 0$. By the continuity of $f$, the local boundedness of $\partial f$ and the isometric property of the vector transport $\mathcal T$, we have $\delta^i\rightarrow 0$ as $i\rightarrow\infty$. Since the line-search does not terminate, we obtain
   	$-\delta^i+\beta_{t^i d_k}\langle \mathcal{T}_{x_k\leftarrow R_{x_k}(t^i d_k)}(g^i),d_k\rangle< -\theta_{\rm R}w_k$ for all  large $i$. 
   	Therefore, by the continuity of $\beta$, it follows that
   \begin{align*}
     \limsup_{i\in \mathcal{I}_0}\beta_{t^i d}\langle \mathcal{T}_{x_k\leftarrow R_{x_k}(t^i d_k)}(g^i),d_k\rangle & = \limsup\limits_{i\in \mathcal {I}_0}\beta_{t^*d}\langle \mathcal{T}_{x_k\leftarrow R_{x_k}(t^i d_k)}(g^i),d_k\rangle\\
      &  \leq -\theta_{\rm R} w_k<-\theta_{\rm T} w_k,
   \end{align*}
   which is a contradiction with (\ref{linelem2}).
   	
   	Second, suppose that $t^*>0$. From (\ref{linelem1}), $\theta_{\rm L}<\theta_{\rm T}$ and the continuity of $f\circ R_{x_k}$, for sufficiently large $i$, we obtain
   	$$ f(R_{x_k}(t^i d_k))\leq f(x_k)-\theta_{\rm L}t^i w_k.$$
   Since the line-search does not terminate, it follows that, for sufficiently large $i$, $\delta^i\leq \theta_{\rm A}w_k$ from Steps \ref{LS:8}-\ref{LS:10} of Procedure \ref{Procedure:LS} and that
   	$$\beta_{t^i d_k}\langle\mathcal{T}_{x_k\leftarrow R_{x_k}(t^i d_k)}(g^i),d_k\rangle <-\theta_{\rm R} w_k+\delta^i\leq (\theta_{\rm A}-\theta_{\rm R})w_k< -\theta_{\rm T} w_k$$
   from Steps \ref{LS:11}-\ref{LS:13} of Procedure \ref{Procedure:LS}. This contradicts (\ref{linelem2}).\qed
   \end{proof}
   \section{Global convergence}\label{sec4}
   In this section, we establish the convergence of Algorithm \ref{Algo:RQNBM},
   i.e., if the number of serious iteration steps is finite, then the last serious iteration is stationary; otherwise, every accumulation point of the serious iteration sequence is stationary. In the convergence analysis, the tolerance is set to be $\varepsilon= 0$.
    
  The convergence analysis relies on Assumption \ref{assum3}, along with the following two additional assumptions.
  Assumption~\ref{assum4} is introduced to ensure that $R^{-1}_{x_k}(y_{k+1})$ exists and is unique in a neighborhood of $x_k$. 
   
   
   \begin{assumption}\label{assum4}
       The injectivity radius of $\mathcal{M}$ is greater than $0$, i.e., ${\rm Inj}_R(\mathcal{M}) >0$.
   \end{assumption}
   
   \begin{assumption}\label{assum1}
   The sublevel set $\mathcal{N}= \{x \in \mathcal{M} : f(x) \leq f(x_1)\}$ is bounded.
   \end{assumption}
   
   The following three lemmas present some necessary results for our analysis. In particular, we use $\|R_{x_{k+1}}^{-1}(y_{k+1})\|$ instead of $\|y_{k+1}-x_{k+1}\|$ in the Euclidean setting, because $\|y_{k+1}-x_{k+1}\|$ is not defined in the Riemannian setting.
   \begin{lemma}\label{lemma4.1} 
   If Assumptions~\ref{assum3} and \ref{assum4} hold, then at the $k$th iteration of Algorithm \ref{Algo:RQNBM}, we have
   	\begin{gather}
   		w_k=\langle\tilde{g}_k,\mathcal{H}_k\tilde{g}_k\rangle
   		+2\tilde{\alpha}_k\geq \rho\|\tilde{g}_k\|^2, \label{wk}\\
   		\alpha_{k+1}\geq \gamma (\|R_{x_{k+1}}^{-1}(y_{k+1})\|)^\nu.\label{alphak}
   	\end{gather}
   \end{lemma}
   \begin{proof}
   From Steps \ref{Step:30}-\ref{Step:36} of Algorithm \ref{Algo:RQNBM} and the positive-definiteness of $\check{\mathcal{H}}_{k}$, we can see $\check{w}_{k} =\langle\tilde{g}_{k},\check{\mathcal{H}}_{k}\tilde{g}_{k}\rangle+2\tilde{\alpha}_{k}\geq 0.$ If $\check{w}_{k}<\rho\|\tilde{g}_{k}\|^2$ or $\check{w}_{k}\geq\rho\|\tilde{g}_{k}\|^2$ and $i_C=i_U=1$ , then
   $$w_{k}=\check{w}_{k}+\rho\|\tilde{g}_{k}\|^2 = \langle\tilde{g}_{k},(\check{\mathcal{H}}_{k}+ \rho {\rm {id}}_{T_{x_k}\mathcal{M}})\tilde{g}_{k}\rangle+2\tilde{\alpha}_{k}=\langle\tilde{g}_k,\mathcal{H}_k\tilde{g}_k\rangle
   		+2\tilde{\alpha}_k$$
     and $w_{k}\geq\rho \|\tilde{g}_k\|^2$. Otherwise,
   $$w_{k}=\check{w}_{k}= \langle\tilde{g}_{k},\check{\mathcal{H}}_{k}\tilde{g}_{k}\rangle+2\tilde{\alpha}_{k}=\langle\tilde{g}_k,\mathcal{H}_k\tilde{g}_k\rangle
   		+2\tilde{\alpha}_k\geq \rho \|\tilde{g}_k\|^2.$$
   Thus,  (\ref{wk}) holds.
   
   For serious steps, taking into account that $\|R_{x_{k+1}}^{-1}(y_{k+1})\|=0$ and $\alpha_{k+1}=0$, (\ref{alphak}) holds trivially. For null steps,  (\ref{alphak}) follows from  (\ref{deltak}) and $x_k=x_{k+1}$. \qed
   \end{proof}
   
   \begin{lemma}\label{lambda1}
    Under Assumptions~\ref{assum3} and \ref{assum4}, if $k$ is the current iteration index of Algorithm \ref{Algo:RQNBM},  there exist scalars $\Lambda^k_i\geq 0, i=m,\ldots, k$ with $\sum_{i=m}^{k}\Lambda^k_i=1$ such that
    \begin{equation}\label{sum1}
   (\tilde{g}_k,\tilde{\alpha}_k)=\sum_{i=m}^{k}\Lambda^k_i(\hat{\mathcal{T}}_{x_{k}\leftarrow y_{i}}{(g_i)},\alpha_i).
    \end{equation}
   Moreover,  defining $\tilde{\sigma}_k=\sum_{i=m}^{k}\Lambda_i^k\|R_{x_{k}}^{-1}(y_{i})\|$, we have $\tilde{\alpha}_k\geq\gamma\tilde{\sigma}_k^\nu$.
   	
   \end{lemma}
   \begin{proof}
   From the definition of index $m$, it follows that $x_i=x_m$ for $i=m,\ldots, k$. We prove \eqref{sum1} by induction.
   	
   For $k=m$, we know that $x_m=y_m$, $\tilde{g}_m=g_m$, $\tilde{\alpha}_m=\alpha_m=0$. Thus by setting  $\Lambda_m^m=1$, we obtain (\ref{sum1}). Next we assume that (\ref{sum1}) holds for $k~(\geq m)$, and will show that  (\ref{sum1}) is also true for $k+1$.
   Define	$\Lambda_m^{k+1}=\lambda_{k}^{(1)}+\lambda_{k}^{(3)}\Lambda_{m}^k$,
   	$\Lambda_i^{k+1}=\lambda_{k}^{(3)}\Lambda_{i}^k,\ \
   	\Lambda_{k+1}^{k+1}=\lambda_{k}^{(2)},$
   for $i=m+1,\ldots,k$. Obviously, $\Lambda_i^{k+1}\geq 0$ for  all $i=m,\ldots,k+1$ and $$\sum\limits_{i=m}^{k+1}\Lambda_i^{k+1}=\lambda_{k}^{(1)}+\lambda_{k}^{(3)}(\Lambda_{m}^k+\sum\limits_{i=m+1}^{k}\Lambda_{i}^k)+\lambda_{k}^{(2)}
   	=1.$$
   This together with (\ref{aggre1}) and (\ref{aggre2}) shows that
   	\begin{equation*}
   		\begin{split}
   			(\tilde{g}_{k+1},\tilde{\alpha}_{k+1})&=\lambda_{k}^{(1)}(g_m,0)+\lambda_{k}^{(2)}
   			(\hat{\mathcal{T}}_{x_{k+1}\leftarrow y_{k+1}}(g_{k+1}),\alpha_{k+1})+	\sum\limits_{i=m}^{k}\lambda_{k}^{(3)}\Lambda_{i}^k(\hat{\mathcal{T}}_{x_{k}\leftarrow y_{i}}(g_i),\alpha_i)  \\
   			& =\sum\limits_{i=m}^{k+1}\Lambda_{i}^{k+1}(\hat{\mathcal{T}}_{x_{k+1}\leftarrow y_{i}}(g_i),\alpha_i),
   		\end{split}
   	\end{equation*}
   	which implies that (\ref{sum1}) holds by replacing $k$ with $k+1$.
   
   	Moreover, since $x_i=x_k$ for $i=m,\ldots, k$, we have
   	$\tilde{\sigma}_k=\sum_{i=m}^{k}\Lambda_i^k\|R_{x_{i}}^{-1}(y_{i})\|.$
   	This along with (\ref{alphak}) and (\ref{sum1}) shows that
   	\begin{equation*}
   		\gamma\tilde{\sigma}_k^\nu= \gamma\left(\sum_{i=m}^{k}\Lambda_i^k\|R_{x_{i}}^{-1}(y_{i})\|\right)^\nu \leq \sum_{i=m}^{k}\Lambda_i^k\gamma (\|R_{x_{i}}^{-1}(y_{i})\|)^\nu \leq \sum_{i=m}^{k}\Lambda_i^k\alpha_i = \tilde{\alpha}_k,
   	\end{equation*}
   	where the first inequality is due to the convexity of the function $z\ \mapsto\ \gamma z^\nu$ on $\mathbb{R}_+$
   	for $\gamma>0$ and $w\geq 1$. \qed
   \end{proof}

   \begin{lemma}\label{station1}
   	Suppose that Assumption~\ref{assum4} holds. Given $\bar{x}\in \mathcal{M}$ and a nonempty finite index set $\bar{\mathcal{I}}$, if there exist ${\bar q}$, ${\bar g}_i$, ${\bar y}_i$ and scalars ${\bar \Lambda}_i\geq 0$, for $i\in \bar{\mathcal{I}}$ with $\sum_{i\in \bar{\mathcal{I}}}{\bar \Lambda}_i=1$, such that
   	\begin{equation}\label{station2}
   		({\bar q},0)= \sum_{i\in\bar{\mathcal{I}}}{\bar \Lambda}_i(\hat{\mathcal{T}}_{ \bar x\leftarrow {\bar y}_{i}}({\bar g}_i), \|R_{\bar x}^{-1}(\bar y_{i})\|),
   	\end{equation}
   then ${\bar q}\in \partial f({\bar x})$, where ${\bar g}_i\in\partial f({\bar y}_i)$.
   \end{lemma}
   \begin{proof} 
Define $\bar{\mathcal{I}}_{+}=\{i\in\bar{\mathcal{I}}:\ {\bar \Lambda}_i>0\}$.
   	From (\ref{station2}), it follows that ${\bar y}_i={\bar x}\ \ {\rm and}\ \  {\bar g}_i\in \partial f({\bar x})$, ${\rm for\ all}\  i\in \bar{\mathcal{I}}_{+}$. 
    Thus, we obtain
   	$\beta_{R_{\bar x}^{-1}(\bar y_{i})} = \beta_{0_{\bar x}} = 1$, which implies that $\hat{\mathcal{T}}_{ \bar x\leftarrow {\bar y}_{i}}$, $i\in \bar{\mathcal{I}}_{+}$ are identity mappings.
    Therefore,
     $$ {\bar q}= \sum_{i\in \bar{\mathcal{I}}_{+}}{\bar \Lambda}_i{\bar g}_i,\ \hbox{with}~ {\bar \Lambda}_i>0, \  i\in \bar{\mathcal{I}}_{+}~\hbox{and}  \ \sum_{i\in \bar{\mathcal{I}}_{+}}{\bar \Lambda}_i=1.$$ 
     Combining this with the convexity of $\partial f({\bar x})$ (Theorem \ref{nonsmo-theo}) implies that ${\bar q}\in \partial f({\bar x})$. 
     
     \qed
    


   
   \end{proof}
   
   Based on the above lemmas, we prove in the following theorem that, if Algorithm \ref{Algo:RQNBM} stops finitely at iteration $k$ in Step \ref{Step:stoping}, then the current stability center $x_k$ is stationary for the objective function $f$.
   
   \begin{theorem}\label{Thm:finite-stop}
   	Suppose that Assumptions~\ref{assum3} and \ref{assum4} hold and that Algorithm \ref{Algo:RQNBM} terminates finitely at the $k$th iteration. Then $x_{k}$ is a stationary point of $f$, i. e., we have $0_{x_{k}}\in \partial f(x_{k})$.	
   \end{theorem}
   
   \begin{proof}
   	If Algorithm \ref{Algo:RQNBM} terminates at the $k$th iteration, then $w_{k}=0$.
   	From (\ref{wk})  and Lemma \ref{lambda1}, we have
   	$\tilde{g}_{k}=0_{x_k}$, $\tilde{\alpha}_{k}=\tilde{\sigma}_{k}=0$ and
   $(0_{x_k},0)=\sum_{i=m}^{k}\Lambda^k_i(\hat{\mathcal{T}}_{x_{k}\leftarrow y_{i}}(g_i),\|R_{x_{k}}^{-1}(y_{i})\|).$
   It follows from ${\bar x}=x_{k}, \ \ {\bar q} =0_{x_k},\ \
   			{\bar g}_i=g_i,\ \   {\bar y}_i=y_i,\ \  {\bar \Lambda}_i=\Lambda_i^k$ and $\bar{\mathcal{I}} = \{m,\ldots,k\}$ in Lemma~\ref{station1} that $0_{x_k} \in \partial f({x_k})$. \qed
   \end{proof}

   \vspace{3mm}
   In what follows, we assume that Algorithm \ref{Algo:RQNBM} does not terminate.
   Two cases are considered separately: (a) the number of serious steps is finite followed by an infinite number of null steps; and (b) the number of serious steps is infinite.
   
   Now we consider the former case and first prove the following fundamental lemma. The proof of the lemma depends heavily on the fact that the vector transport $\mathcal{T}$ and $\beta$ are continuous.
   \begin{lemma}\label{lemma4.5}
   		
   	 If  Assumptions~\ref{assum3}-\ref{assum1} hold and there exist a point $x^*$ and an infinite index set $\mathcal{K}$ such that the sequence $\{x_k\}_\mathcal{K}$ converges to $x^*$ and $\{w_k\}_\mathcal{K}$ converges to $0$, then $0_{x^*}\in \partial f(x^*)$.
   \end{lemma}
   \begin{proof}
      First, we generalize an important conclusion that emerged in \cite[Rem. 2.1]{hoseini2021proximal}. If there exists an infinite index set $\mathcal{I}\subseteq\{1,2,\ldots\}$ satisfying $(x_i,y_i,\xi_i) \rightarrow ({x^*},{y^*},\xi^*)$ as $i\rightarrow\infty$, $i\in\mathcal{I}$ and a sequence $\{\eta_i\}$ bounded such that $R_{x_i}^{-1}(y_i)=\eta_i$, then
      $$\lim\limits_{i\in\mathcal{I}}\mathcal{T}_{x_i\rightarrow y_i}(\xi_i) =\mathcal{T}_{x^*\rightarrow {y^*}}({\xi^*}).$$
      Since $R^{-1}$ is continuous, it follows that $\lim\limits_{i\in\mathcal{I}}\eta_i=\lim\limits_{i\in\mathcal{I}}R_{x_i}^{-1}(y_i)=R_{x^*}^{-1}(y^*)$. Using the continuity of $\mathcal{T}$, one has $$\lim\limits_{i\in\mathcal{I}}\mathcal{T}_{x_i\rightarrow y_i}(\xi_i)=\lim\limits_{i\in\mathcal{I}}\mathcal{T}_{R_{x_i}^{-1}(y_i)}(\xi_i)=\mathcal{T}_{R_{{x^*}}^{-1}({y^*})}({\xi^*})=\mathcal{T}_{{x^*}\rightarrow {y^*}}({\xi^*}).$$
   

   	Next, we establish the boundedness of $\{y_k\}$ and $\{g_k\}$.
    Since the level set $\mathcal N$ is bounded, we know that $\{x_k\}$ is bounded. Since $y_{k+1} \in {\rm cl}{R_{x_k}(B(0_{x_k}, \mu_0))}$ for all $k$, $\{y_k\}$ is also bounded.
   It then follows from Theorem \ref{nonsmo-theo} that the sequence  $\{g_k\}$ is bounded.
   In addition, from Lemma \ref{lambda1} and the Carath\'{e}odory theorem, we can conclude that there exists some index set
   $\mathcal{I}_k\subseteq\{1,2,\ldots,n+2\}$ and
   $y_{k,i}$, $g_{k,i}\in \partial f(y_{k,i})$, $\tilde{\Lambda}^{k}_i\geq 0$,  $i\in \mathcal{I}_k$ such that $\sum_{i\in \mathcal{I}_k}{\tilde{\Lambda}}^{k}_i=1$ and
           \begin{align*}
   		&(\tilde{g}_k,\tilde{\sigma}_k)=\sum_{i\in \mathcal I_k}\tilde{\Lambda}^{k}_i(\hat{\mathcal{T}}_{x_k\leftarrow {y}_{k,i}}(g_{k,i}),\|R_{x_k}^{-1}(y_{k,i})\|),\ \ \hbox{for\ all}~k\geq 1,
   	\end{align*}
   	where
   	$(y_{k,i},g_{k,i})\in \{(y_j,g_j):\ j=m,\ldots, k\}.$
       Since there are only finite ways to take $\mathcal{I}_k$, there must exist an infinite index set $\tilde{\mathcal{K}}\subseteq \mathcal{K}$ such that $\mathcal{I}_k=\mathcal{I}_0$ for all $k\in \tilde{\mathcal{K}}$, where $\mathcal{I}_0\subseteq\{1,2,\ldots,n+2\}$ is a fixed index set.
       Then we obtain
        \begin{equation}\label{Lem4.4-1}
       		(\tilde{g}_k,\tilde{\sigma}_k)=\sum_{i\in \mathcal I_0}\tilde{\Lambda}^{k}_i(\hat{\mathcal{T}}_{x_k\leftarrow {y}_{k,i}}(g_{k,i}),\|R_{x_k}^{-1}(y_{k,i})\|), \ \ \hbox{for\ all}~k\in \tilde{\mathcal{K}}.
       \end{equation}
   Since $\{y_{k,i}\}_{k=1}^{\infty}$ is bounded for each $i\in \mathcal{I}_0$, there must be an infinite index set $\mathcal{K}_1\subseteq\tilde{\mathcal{K}}$ and points $y^*_i$, $i\in \mathcal I_0$ such that $y_{k,i}{\rightarrow} y^*_i$ as $k\rightarrow\infty$, $k\in \mathcal{K}_1$, for all $i\in \mathcal{I}_0$. Similarly, by using the boundedness of $\{g_k\}$ and $\{\tilde{\Lambda}^{k}_i\}$, as well as the continuity of $\beta$, it follows from Theorem \ref{nonsmo-theo} (iii) that
   there exist $g^*_i\in \partial f(y^*_i)$, $\tilde{\Lambda}^*_i\geq 0$, $i\in \mathcal I_0$ and an infinite index set $\mathcal K_0\subseteq\mathcal{K}_1$ such that
   \begin{equation*}
   g_{k,i}\rightarrow g^*_i, \ \tilde{\Lambda}_{i}^k\rightarrow \tilde{\Lambda}^*_i\  {\rm and}\ \beta_{R_{x_k}^{-1}(y_{k,i})}\rightarrow \beta_{R_{x^*}^{-1}(y_{i}^*)},
   \end{equation*}
   as $k\rightarrow\infty$, $k\in \mathcal K_0$, for $i\in \mathcal I_0$. It is obvious that $\sum_{i\in \mathcal{I}_0}\tilde{\Lambda}^{*}_i=1$.  Since $R^{-1}$ is continuous, it follows that $\lim\limits_{k\in\mathcal{K}_0}R_{x_k}^{-1}(y_{k,i})=R_{x^*}^{-1}(y_i^*)$. Using the continuity of $\mathcal{T}$, one has $$\lim\limits_{k\in\mathcal{K}_0}\mathcal{T}_{x_k\leftarrow {y}_{k,i}}(g_{k,i})=\lim\limits_{k\in\mathcal{K}_0}\mathcal{T}_{R_{x_k}^{-1}(y_{k,i})}(g_{k,i})=\mathcal{T}_{R_{{x^*}}^{-1}({y_i^*})}({g_i^*})=\mathcal{T}_{{x}^*\leftarrow {y}_{i}^*}(g^*_i).$$
   Thus
   \begin{equation}\label{Lem4.4-2}
       \lim\limits_{k\in\mathcal{K}_0}\hat{\mathcal{T}}_{x_k\leftarrow {y}_{k,i}}(g_{k,i})=\lim\limits_{k\in\mathcal{K}_0}\beta_{R_{x_k}(y_{k,i})}{\mathcal{T}}_{x_k\leftarrow {y}_{k,i}}(g_{k,i})= \beta_{R_{x^*}^{-1}(y_{i}^*)}\mathcal{T}_{{x}^*\leftarrow {y}_{i}^*}(g^*_i)=\hat{\mathcal{T}}_{{x}^*\leftarrow {y}_{i}^*}(g^*_i).
   \end{equation}
   
       By (\ref{wk}), Lemma \ref{lambda1} and $\lim\limits_{k\in \mathcal{K}}w_{k}=0$, we have
   	\begin{equation}\label{Lem4.4-3}
    \tilde{g}_{k}\rightarrow 0_{x^*},\ \ \tilde{\alpha}_{k}\rightarrow 0,\ \ \tilde{\sigma}_{k}\rightarrow 0\ \ {\rm as} \ \ k\rightarrow\infty,\ \ k\in \mathcal K_0.
    \end{equation}
    Thus, by taking the limit as  $k\in \mathcal K_0$ and $k\rightarrow\infty$ in \eqref{Lem4.4-1}  and using \eqref{Lem4.4-2} and \eqref{Lem4.4-3}, we have
   $$
          (0_{x^*},0)=\sum_{i\in \mathcal I_0}\tilde{\Lambda}^{*}_i(\hat{\mathcal{T}}_{x^*\leftarrow {y}_{i}^*}(g_{i}^*),\|R_{x^*}^{-1}(y_{i}^*)\|).
   $$
   This together with Lemma \ref{station1} shows that $0_{x^*}\in \partial f({x}^*)$. \qed
   \end{proof}
   
   \begin{lemma}\label{lemma4.6}\cite[Lem. 3.5]{lukvsan1999globally}
        Let the vectors $p$, $q$ and the scalars $w\geq 0$, $\alpha \geq 0$, $\chi \geq 0$, $K\geq 0$, $c\in (0, 1/2)$ satisfy
   	$$w=\|p\|^2+2\alpha,\qquad \chi +\langle p,q\rangle\leq c w,\qquad \max\{\|p\|, \|q\|, \sqrt{\alpha}\}\leq K.$$
   	Let
   	$ 
   		Q(\lambda)=\|\lambda q+(1-\lambda)p\|^2+2(\lambda \chi+(1-\lambda)\alpha)$ and $b=\frac{1-2c}{4K}.
   	$ 
   	Then $\min\{Q(\lambda)\ \vert\ \lambda\in [0, 1]\}\leq w-w^2b^2.$
   \end{lemma}
   
   Now we present the first convergence result of Algorithm \ref{Algo:RQNBM}, 
   which states that if the number of serious steps is finite, then the last serious iterate is stationary. The proof of this theorem is an extension of the Euclidean case in \cite[Lem. 3.6]{vlvcek2001globally}. 
   The main difference between the Euclidean setting and the Riemannian setting occurs when $x_{k+1}$ is not equal to $x_k$. In the former setting $\tilde{\mathcal{H}}_{k}$ is equal to $\mathcal{H}_k$. However, $\tilde{\mathcal{H}}_{k}$ is different from $\mathcal{H}_k$ since the vector transport $\mathcal{T}$ is applied in our method.
   
   \begin{theorem}\label{Thm:4.2}
   Suppose that Assumptions~\ref{assum3}-\ref{assum1} hold and Algorithm \ref{Algo:RQNBM} generates a finite number of serious steps followed by an infinite number of null steps.
   Let $k^*$ be the index of the last serious iterate, then it follows that
    $x_{k^*}$ is a stationary point of $f$ on $\mathcal M$, i.e., $0_{x_{k^*}}\in \partial f(x_{k^*})$.
   \end{theorem}
   \begin{proof}
   Use ${{H}}_{k}$ and ${\tilde{{H}}}_{k}$ to represent the coordinate representation of operators ${\mathcal {H}}_{k}$ and $\tilde{\mathcal {H}}_{k}$, respectively.\footnote{Let $(U, \varphi)$ be a chart of the manifold $\mathcal{M}$ with $x\in U$. If $\mathcal{A}$ is an operator on $T_x\mathcal{M}$, its coordinate representation is defined by $A=({\rm D}\varphi(x))\circ\mathcal{A}\circ({\rm D}\varphi(x))^{-1}$.} We use ${\rm Tr}(H_k)$ to denote the trace of the matrix $H_k$, which is independent of the chosen basis \cite[Lem. 3.11]{huang2015broyden}. Since $k^*$ is the index of the last serious iterate, we have that $x_{k}= x_{k^*}$ for all $k\geq k^*$.
   
      Firstly, we show that there exists an index $k_0$ $(>k^*)$ such that
          \begin{equation}\label{wk+1}
              w_{k+1}\leq\langle\tilde{g}_{k+1},\tilde{\mathcal {H}}_{k}\tilde{g}_{k+1}\rangle+2\tilde{\alpha}_{k+1},\qquad \forall\ k\geq k_0,
          \end{equation}
   		\begin{equation}\label{trace1}
   		 {\rm Tr}({{H}}_{k+1})\leq {\rm Tr}({{H}}_k),\qquad \forall\ k\geq k_0.
   		\end{equation}
   Since the scaling of $\check{\mathcal{H}}_{k+1}$ in Step \ref{Step: scaling}
   of Algorithm \ref{Algo:RQNBM} does not impair the conclusions (\ref{wk+1}) and (\ref{trace1}) obtained, in the following analysis we will neglect this step. The proof is divided into two cases as follows.
   
   	{\bf Case 1}. If $n_{\rm C}\geq \Gamma$ is executed, let $k_1$ denote the index of the first execution. Then $i_{\rm C}=1$ for all $k\geq k_1$. Next, we prove that $\mathcal{H}_k- \rho{\rm id}_{T_{x_k}\mathcal{M}}$ is positive-definite for all $k>k_1$ by induction.
   
    From Step \ref{Step: correction} of Algorithm \ref{Algo:RQNBM} and the positive-definiteness of $\check{\mathcal{H}}_{{k_1+1}}$, we conclude that $\mathcal H_{{k_1+1}}-\rho {\rm id}_{T_{x_{k_1+1}}\mathcal{M}}$ is positive-definite. Assuming that $\mathcal{H}_{k}- \rho {\rm id}_{T_{x_k}\mathcal{M}}$ is positive-definite for $k\geq k_1+1$, we will show that $\mathcal{H}_{k+1}- \rho{\rm id}_{T_{x_{k+1}}\mathcal{M}}$ is positive-definite. According to Algorithm \ref{Algo:RQNBM}, if $\check{w}_{k+1}<\rho\|\tilde{g}_{k+1}\|$ or $\check{\mathcal H}_{k+1}$ is updated by SR1 or BFGS, we get $i_{\rm C}=i_{\rm U}=1$ and $\mathcal{H}_{k+1}- \rho{\rm id}_{T_{x_{k+1}}\mathcal{M}}=\check{\mathcal H}_{k+1}$;
    otherwise $\mathcal{H}_{k+1}- \rho{\rm id}_{T_{x_{k+1}}\mathcal{M}}=\check{\mathcal{H}}_{k+1}-\rho {\rm id}_{T_{x_{k+1}}\mathcal{M}}=\tilde{\mathcal{H}}_k- \rho{\rm id}_{T_{x_{k+1}}\mathcal{M}}=\mathcal{T}_{x_k\rightarrow x_{k+1}}\circ(\mathcal{H}_k- \rho{\rm id}_{T_{x_{k}}\mathcal{M}})\circ\mathcal{T}_{x_k\leftarrow x_{k+1}}$. Because of the positive-definiteness of $\mathcal{H}_{k}- \rho{\rm id}_{T_{x_{k}}\mathcal{M}}$, we know that $\mathcal{H}_{k+1}- \rho{\rm id}_{T_{x_{k+1}}\mathcal{M}}$ is positive-definite.
    In conclusion, $\mathcal{H}_k- \rho{\rm id}_{T_{x_{k}}\mathcal{M}}$ is positive-definite for all $k> k_1$.
   
   Define
   	$k_0=\max\{{k}_1,{k}^*\}+1$.
   Thus, the update for serious step (Steps \ref{Step:20}-\ref{Step:25}) of Algorithm \ref{Algo:RQNBM} is not performed for all $k\geq k_0$.
   If the SR1 update is not used, then $i_{\rm U}=0$ and $\check{\mathcal{H}}_{k+1}=\tilde{\mathcal{H}}_k$.
   This implies that $\check{\mathcal{H}}_{k+1}- \rho{\rm id}_{T_{x_{k+1}}\mathcal{M}}\ (=\tilde{\mathcal{H}}_k-\rho {\rm id}_{T_{x_{k+1}}\mathcal{M}}={\mathcal{H}}_k- \rho{\rm id}_{T_{x_{k}}\mathcal{M}})$ is positive-definite. Then it follows that $\langle\tilde{g}_{k+1},(\check{\mathcal{H}}_{k+1}- \rho{\rm id}_{T_{x_{k+1}}\mathcal{M}})\tilde{g}_{k+1}\rangle\geq 0$,
   and therefore $$\check{w}_{k+1}=\langle\tilde{g}_{k+1},\check{\mathcal{H}}_{k+1}\tilde{g}_{k+1}\rangle+2\tilde{\alpha}_{k+1}\geq \rho\|\tilde{g}_{k+1}\|^2+2\tilde{\alpha}_{k+1}\geq\rho\|\tilde{g}_{k+1}\|^2.$$
    Thus, from Step \ref{Step:no-correction} of Algorithm \ref{Algo:RQNBM}, we have $w_{k+1}=\check{w}_{k+1}$, $\mathcal H_{k+1}=\check{\mathcal{H}}_{k+1}=\tilde{\mathcal{H}}_k$. This together with the fact that ${\rm Tr}({\tilde{{H}}}_{k})= {\rm Tr}({{H}}_k)$ (see  \cite[Lem. 3.11]{huang2015broyden}) shows that  (\ref{wk+1}) and (\ref{trace1}) are satisfied.
   
   On the other hand, if the SR1 update is executed, then the conditions $\langle\tilde{g}_{k},v_k\rangle<0$ and (\ref{SR1cond1}) are satisfied,
   as well as $i_{\rm C}=i_{\rm U}=1$. Hence, Step \ref{Step: correction} is executed, which
    along with (\ref{SR1update}) shows that
        \begin{align*}
            w_{k+1}&=\check{w}_{k+1}+ \rho\|\tilde{g}_{k+1}\|^2\\  &=\langle\tilde{g}_{k+1},\check{\mathcal{H}}_{k+1}\tilde{g}_{k+1}\rangle+2\tilde{\alpha}_{k+1} + \rho\|\tilde{g}_{k+1}\|^2\\
         &= \langle\tilde{g}_{k+1},(\mathcal{\tilde{H}}_k-\frac{v_kv_k^\flat}{\tilde{u}_k^\flat v_k})\tilde{g}_{k+1}\rangle+2\tilde{\alpha}_{k+1} + \rho\|\tilde{g}_{k+1}\|^2\\
         &= \langle\tilde{g}_{k+1},\tilde{\mathcal{H}}_k\tilde{g}_{k+1}\rangle+2\tilde{\alpha}_{k+1}+ \rho \|\tilde{g}_{k+1}\|^2-\frac{\langle\tilde{g}_{k+1},v_k\rangle^2}{\langle \tilde{u}_k,v_k\rangle}
        \end{align*}
   	and
   	$${\rm Tr}({H}_{k+1})={\rm Tr}({{{H}}}_k)+\rho n-\frac{\|v_k\|^2}{\langle\tilde{u}_k,v_k\rangle}.$$
    Therefore, it follows from (\ref{SR1cond1}) that  (\ref{wk+1}) and (\ref{trace1}) hold.
   	
   {\bf Case 2}.	If $n_{\rm C}<\Gamma$ for all $k\geq 1$, let $\bar{k}$ denote the index of the last change in $n_{\rm C}$. Set $k_0=\max\{\bar{k},k^*\}+1$, then we have $w_{k+1}=\check{w}_{k+1}$ and $\mathcal H_{k+1}=\check{\mathcal{H}}_{k+1}$ for all $k> k_0$. If the SR1 update is used, we have
   $$w_{k+1}=\langle\tilde{g}_{k+1},\check{\mathcal{H}}_{k+1}\tilde{g}_{k+1}\rangle+2\tilde{\alpha}_{k+1}=\langle\tilde{g}_{k+1},\tilde{\mathcal{H}}_k\tilde{g}_{k+1}\rangle
   	+2\tilde{\alpha}_{k+1}-\frac{\langle\tilde{g}_{k+1},v_k\rangle^2}{\langle \tilde{u}_k,v_k\rangle}$$
   	and
   	$${\rm Tr}({H}_{k+1})={\rm Tr}({{{H}}}_k)-\frac{\|v_k\|^2}{\langle\tilde{u}_k,v_k\rangle}.$$
    Since $\langle\tilde{u}_k,v_k\rangle>0$, (\ref{wk+1}) and (\ref{trace1}) can be obtained. Otherwise (the SR1 update is not used), it is clear that (\ref{wk+1}) and (\ref{trace1}) hold with equalities.

   Secondly, we prove the conclusion that $w_k\rightarrow 0$.
        Combining (\ref{mainsubprob})-(\ref{aggre2}) with (\ref{wk}) and (\ref{wk+1}), we obtain
           \begin{equation}\label{equ:31}
               \begin{split}
           w_{k+1} & \leq \langle\tilde{g}_{k+1},\tilde{\mathcal{H}}_k\tilde{g}_{k+1}\rangle
           +2\tilde{\alpha}_{k+1} \\
           & =\|\lambda_{k}^{(1)}\mathcal{W}_kg_m+\lambda_{k}^{(2)}\mathcal{W}_k\hat{\mathcal{T}}_{x_k\leftarrow y_{k+1}}(g_{k+1})+\lambda_{k}^{(3)}\mathcal{W}_k\tilde{g}_k\|^2+2(\lambda_{k}^{(2)}\alpha_{k+1}+\lambda_{k}^{(3)}\tilde{\alpha}_k) \\
           &=\varphi(\lambda_{k}^{(1)},\lambda_{k}^{(2)},\lambda_{k}^{(3)})\\
           &\leq\varphi(0,0,1)\\
           &=w_k, \ \ \ \hbox{for\ all}\ k\geq k_0.
            \end{split}
           \end{equation}
   Hence the sequence $\{w_k\}$ is bounded and convergent. From (\ref{wk}), we see that $\tilde{\alpha}_k\leq w_k$, and therefore $\{\tilde{\alpha}_k\}$ is bounded. In addition, using Lemma \ref{lemma4.1} and (\ref{trace1}), we obtain the boundedness of $\{\mathcal W_k\tilde{g}_k\}$, $\{\mathcal{H}_k\}$ and $\{\mathcal W_k\}$. Since ${\rm D}R_{x_{k^*}}$ is bounded on ${\rm cl}{B}(0_{x_{k^*}}, \mu_0)$, there is a real $\mu_1>0$ such that $\beta_{\eta}\leq \mu_1$ for all $\eta\in{\rm cl}{B}(0_{x_{k^*}},\mu_0)$.  Then the sequence $\{\mathcal W_k\hat{\mathcal{T}}_{x_{k^*}\leftarrow {y}_{k+1}}(g_{k+1})\}$ is bounded by the boundedness of $\{\mathcal{T}_{x_{k^*}\leftarrow {y}_{k+1}}(g_{k+1})\}$.

     Moreover, it follows that
   	\begin{align*}
   		& \min\{\varphi(\lambda^{(1)},\lambda^{(2)},\lambda^{(3)})\ \vert\ \lambda^{(i)}\geq 0,i=1,2,3, \ \sum_{i=1}^{3}\lambda^{(i)}=1\} \leq \min_{\lambda\in[0,1]}\varphi(0,\lambda,1-\lambda).
   	\end{align*}
       Taking into account that $x_{k}=x_{k^*}$ for all $k\geq k_0$, from (\ref{equ:31}) we have
       \begin{align*}
           	w_{k+1}&\leq \min_{\lambda\in [0,1]} \{\|\lambda \mathcal W_k\hat{\mathcal{T}}_{x_{k^*}\leftarrow {y}_{k+1}}(g_{k+1})+(1-\lambda)\mathcal W_k\tilde{g}_k\|^2+2(\lambda\alpha_{k+1}+(1-\lambda)\tilde{\alpha}_k)\}
       \end{align*}
       for all $k\geq k_0$. Define $M=\sup\{\|\mathcal W_k\hat{\mathcal{T}}_{x_{k^*}\leftarrow {y}_{k+1}}(g_{k+1})\|,\|\mathcal W_k\tilde{g}_k\|,\sqrt{\tilde{\alpha}_k}\ \vert\ k\geq k_0\},\ \  b=\frac{1-2\theta_R}{4M}$. Take
   	\begin{flalign*}
   		\begin{array}{llll}
   			p=\mathcal W_k\tilde{g}_k,\ \  &q=\mathcal W_k\hat{\mathcal{T}}_{x_{k^*}\leftarrow {y}_{k+1}}(g_{k+1}),\ \ & w=w_k,& \\
   			\alpha=\tilde{\alpha}_k, & \chi=\alpha_{k+1}, & c=\theta_R.&
   		\end{array}
   	\end{flalign*}
    This together with (\ref{wk}) and Steps \ref{LS:11}-\ref{LS:13} of Procedure \ref{Procedure:LS} shows that the conditions in Lemma \ref{lemma4.6} are satisfied for all $k\geq k_0$. Thus we obtain
   	$$w_{k+1}\leq w_k-(w_kb)^2,\ \ \forall\ k\geq k_0,$$
   which implies that the sequence $\{w_k\}$ converges to $0$.
   
    Finally, since $\{x_k\}$ converges to $x_{k^*}$,  $0_{x_{k^*}}\in\partial f(x_{k^*})$ by Lemma \ref{lemma4.5}. \qed
   \end{proof}
   
   \vspace{3mm}
   We close this section by considering the latter case where the number of serious steps is infinite.
   \begin{theorem}\label{Thm:4.3}
   Under Assumptions~\ref{assum3}-\ref{assum1}, if Algorithm \ref{Algo:RQNBM} generates an infinite number of serious steps, then every cluster point of $\{x_k\}$ is a stationary point of $f$ on $\mathcal M$.
   \end{theorem}
   \begin{proof}
   	Let ${x}^*$ be a cluster point of $\{x_k\}$, i.e., there is an infinite index set ${\mathcal{K}}$ such that $x_k\rightarrow x^*, \ k\in{\mathcal{K}}$. Define $\mathcal{K}_1=\{k ~ \in \mathbb{N}~ \vert\ \ t^k_{\rm L}>0,\ \exists\ i\in{\mathcal{K}},~i\leq k,~x_i=x_k\}$. It is obvious that $\mathcal{K}_1$ is infinite and $x_k\rightarrow x^*, \ k\in{\mathcal{K}_1}$. Because of the monotonicity of $\{f(x_k)\}$ and the continuity of $f$, it follows that $f(x_k)-f(x_{k+1})\rightarrow 0$ as $k\rightarrow\infty$.
    Thus, from Procedure \ref{Procedure:LS} and the nonnegativity of $t_{\rm L}^k$, we have
   	\begin{equation}\label{theorem4.1}
   		0\leq\theta_{\rm L}t^k_{\rm L}w_k\leq f(x_k)-f(x_{k+1})\rightarrow 0,\ \ k\rightarrow\infty.
   	\end{equation}
   
   Define $\mathcal{K}_2=\{k\in \mathcal{K}_1\ \vert\ t^k_{\rm L}\geq t_{\rm min}\}$. If $\mathcal{K}_2$ is infinite, then $w_k\rightarrow 0$, $x_k \rightarrow {x}^*$ as $k\rightarrow\infty$, $k\in \mathcal{K}_2$ by (\ref{theorem4.1}). Therefore, using Lemma \ref{lemma4.5}, we have the result $0_{x^*}\in\partial f(x^*)$. Otherwise, the set $\mathcal{K}_3=\{k\in \mathcal{K}_1\ \vert\  \delta_{k+1}>\theta_{\rm A}w_k\}$
   	is infinite. Next, we establish that $w_k\rightarrow 0$ as $k\rightarrow\infty$, $k\in \mathcal{K}_3$ by contradiction.
    Suppose that there exists a constant $\gamma_1>0$ and an infinite index set $\mathcal{K}^*\subseteq \mathcal{K}_3$ such that $w_k\geq \gamma_1$ for all $k\in \mathcal{K}^*$.
    Then, from (\ref{theorem4.1}) we have $\lim\limits_{k\in \mathcal{K}^*}t^k_{\rm L}=0$. In addition, from Steps \ref{dk}, \ref{Step: correction} and \ref{Step:no-correction} of Algorithm \ref{Algo:RQNBM}, we obtain
       $$\|t^k_{\rm L}d_k\| \leq t^k_{\rm L}(\|\check{\mathcal{H}}_k\tilde{g}_k\|
   	+ \rho \|\tilde{g}_k\|)\leq t^k_{\rm L}(D+\rho\|\tilde{g}_k\|).$$
     Since the sequence $\{g_k\}$ is bounded,
     from (\ref{aggre1}) we deduce the boundedness of $\{\tilde{g}_k\}$ by induction. Therefore, from the continuity of $R$ and $R^{-1}$, we obtain that
     \begin{equation}\label{equ:33}
         \lim_{k\in \mathcal{K}^*}\|R_{x_k}^{-1}(x_{k+1})\|=\lim_{k\in \mathcal{K}^*}\|t^k_{\rm L}d_k\|=0
     \end{equation}
     and
     \begin{equation}\label{equ:34}
          \lim\limits_{k\in \mathcal{K}^*}x_{k+1}=\lim\limits_{k\in \mathcal{K}^*}R_{x_{k}}(t^k_{\rm L}d_k)=R_{x^*}(0_{x^*})=x^*.
     \end{equation}
     Here $x_{k+1} = y_{k+1}$. From Procedure \ref{Procedure:LS}, we get
       $
            \delta_{k+1}=\max \{\vert f(x_k)-f(x_{k+1})+t^k_{\rm L}\langle\hat{\mathcal{T}}_{x_{k}\leftarrow x_{k+1}}(g_{k+1}),d_k\rangle\vert,\gamma(\|R_{x_k}^{-1}(x_{k+1})\|)^\nu\}$. Combining (\ref{theorem4.1})-(\ref{equ:34}), we have $\lim\limits_{k\in \mathcal{K}^*} \delta_{k+1}= 0$, which contradicts the fact that
   	$$\theta_{\rm A}\gamma_1\leq \theta_{\rm A}w_k<\delta_{k+1},\ \ \forall\ k\in \mathcal{K}^*.$$
   Therefore, it follows that $w_k \rightarrow 0$, $x_k\rightarrow x^*$ as $k\rightarrow\infty$, $k\in \mathcal{K}_3$, which completes the proof by using Lemma \ref{lemma4.5}. \qed
   \end{proof}
   
   \section{Modification with limited-memory BFGS and SR1 updates} \label{sec5}
   
   In Algorithm \ref{Algo:RQNBM}, the update of the operator $\check{\mathcal{H}}_{k}$ by the BFGS and SR1 formulas may lead to computational and storage difficulties for large-scale problems.
   In addition, the calculation of $\mathcal{\tilde{H}}_k=\mathcal{T}_{x_k\rightarrow x_{k+1}}\circ\mathcal{H}_k\circ\mathcal{T}_{x_k\leftarrow x_{k+1}}$ may require matrix multiplication, which also lead to expensive computation. 
    To further reduce the computational cost, we present a modified algorithm (Algorithm \ref{Algo:LRQNBM}) that incorporates the limited-memory BFGS and SR1 update schemes \cite{byrd1994representations,huang2015broyden,huang2015riemannian}.
    
   
   
   \renewcommand{\thealgorithm}{2}
   \begin{algorithm} 
   \caption{Modification with limited-memory BFGS and SR1 updates (M-RQNBM)}
   \label{Algo:LRQNBM}
   \begin{algorithmic}[1]
   \Require  Initial iterate $x_1\in \mathcal{M}$;  initial symmetric positive-definite operator $\mathcal{H}_1$; correction parameters $\tilde{\rho}\in(0, 1)$; stepsize control parameters $t_{\rm min}\in (0,1)$ and $t_{\max}\geq 1$;  constant $0<\mu_0<{\rm Inj}_R(\mathcal{M})$; integer $m_l>0$; length control $\tilde{D}>0$; tolerance $\varepsilon\geq 0$.
   
   \State Set $y_1=x_1$, and compute $g_1\in \partial f(y_1)$. Initialize 
   $\tilde{g}_1=g_1$, $\alpha_1=0$,  $\tilde{\alpha}_1 = 0$,  $m=1$, $l=0$, the discarding indicator $nS = 0$ and $k=1$. \label{Step:1}
   
   \For {$k=1, 2, \ldots$}     
     \State Compute $d_k=-\mathcal{H}_k\tilde{g}_k$, where $\mathcal{H}_k~ (k>1)$ is generated  by (\ref{LBFGS}) if $m = k$ and by (\ref{Lsr1}), otherwise. If $\|d_k\|>\tilde{D}$, then set $d_k = (\tilde{D}/\|d_k\|)d_k$. 
    \Comment{{\it Generate direction}} \label{dk2}
   \State If $nS = 1$,  then discard vector pair $\{s_{k-l-1}^{(k)},u_{k-l-1}^{(k)}\}$ from storage. 
   \State Set $S_k = [s_{k-l}^{(k)},\ldots, s_{k-1}^{(k)}]$, $U_k = [u_{k-l}^{(k)},\ldots, u_{k-1}^{(k)}]$ and $nS=0$.
   
   \State Compute
         $w_{k} =-\langle\tilde{g}_{k},{d}_{k}\rangle+2\tilde{\alpha}_{k}.$ \label{Step1:4}
   \State {\bf If} {$w_k\leq\varepsilon$} {\bf then}
    Return $x_k$. \Comment{{\it Stopping criterion}} \label{Step:stoping2}

   \State Select $t^k_{\rm I}\in \left(0, \min\{t_{\max}, \mu_0/\|d_k\|\}\right]$,
   and call Procedure \ref{Procedure:LS}: 
   \State \hskip 1.2cm $(t_{\rm L}^k, t_{\rm R}^k,\delta_{k+1}, \alpha_{k+1})={\rm LS}(x_k, d_k, w_k, t^k_{\rm I}, t_{\rm min})$.\Comment{{\it Line-search}}\label{Line-search2}
   
    \State Set $x_{k+1}=R_{x_k}(t^k_{\rm L} d_k)$ and $y_{k+1}=R_{x_k}(t^k_{\rm R}d_k)$.  Compute $g_{k+1}\in \partial f(y_{k+1})$. 
    \label{Iteration-Lupdate}
   \If{ $t^k_{\rm L}=0$}  \Comment{{\textit{Update after a null step}}}
   \State	Solve subproblem (\ref{mainsubprob}) to obtain $\lambda_k$. 
   \State Compute $\tilde{g}_{k+1}$ and $\tilde{\alpha}_{k+1}$ by (\ref{aggre1}) and (\ref{aggre2}). \Comment{{\it Subgradient aggregation}} \label{Step1:LSub-agg}
   \State Define $u_k = \mathcal{T}_{x_k\leftarrow y_{k+1}}(g_{k+1})-g_m$, $s_k=t^k_{\rm R}d_k$, $v_k = \mathcal{H}_k u_k - s_k$.
   \If{$\langle\tilde{g}_{k},v_k\rangle<0$  and (\ref{SR1cond1}) holds}
    \label{Step:LSR1-update-condition}
   \Comment{{\it LSR1 update}} \label{Step1:LSR1-update}
   
     \State Set $\tau_{k+1} = 1$. Add $s_k^{(k+1)}$ and $u_k^{(k+1)}$ into storage. If $l\geq m_l$, then 
     \Statex\qquad\qquad\enspace discard vector pair 
    $\{s_{k-l}^{(k)},u_{k-l}^{(k)}\}$ from storage, else $l = {l+1}$. 
    \Statex\qquad\qquad\enspace Update $\tilde{R}_{k+1}$, $C_{k+1}$ and $U_{k+1}^\flat U_{k+1}$.
     
   	
   \Else
   \State Set $\tau_{k+1} = 1$, $S_{k+1} = S_{k}$, $U_{k+1} = U_{k}$.  Update $\tilde{R}_{k+1}$, $C_{k+1}$, $U_{k+1}^\flat U_{k+1}$.
   
   \EndIf
   
   \Else\ ($t^k_{\rm L}>0$)   \Comment{{\bf{\it Update after a serious step}}}
   \State	Set $\tilde{g}_{k+1}=g_{k+1}$, $\tilde{\alpha}_{k+1}=0$, $m=k+1$. \label{Step1:20}\Comment{{\it Aggregation reset}}
   \State Define $u_k = g_{k+1}-\mathcal{T}_{x_k\rightarrow x_{k+1}}(g_m)$,  $s_k=\mathcal{T}_{x_k\rightarrow x_{k+1}}(t^k_{\rm R}d_k)$. \label{Step1:prepare-update1}
   
   \If{$\langle u_k,s_k\rangle>\tilde{\rho}$} \label{Step1:BGFS-update-condition}
   \Comment{{\it LBFGS update}} 
       \State Define $\rho_k = 1/\langle s_k^{(k+1)},u_k^{(k+1)}\rangle$ and $\tau_{k+1} = \langle s_k^{(k+1)},u_k^{(k+1)}\rangle/\|u_k^{(k+1)}\|^2$. 
       \Statex\qquad\qquad\enspace  Add $\{s_k^{(k+1)},u_k^{(k+1)}\}$
   and $\rho_k$ into storage. If $l\geq m_l$, then set $nS = 1$, else 
   \Statex\qquad\qquad\enspace $l = {l+1}$.
        $[s^{(k+1)}_{k-l}, \ldots, s^{(k+1)}_{k-1}]$ = 
       $[\mathcal{T}_{x_k\rightarrow x_{k+1}}(s^{(k)}_{k-l}), \ldots, \mathcal{T}_{x_k\rightarrow x_{k+1}}(s^{(k)}_{k-1})]$,
       \Statex\qquad\qquad\enspace  $[u^{(k+1)}_{k-l}, \ldots, u^{(k+1)}_{k-1}] = [\mathcal{T}_{x_k\rightarrow x_{k+1}}(u^{(k)}_{k-l}), \ldots, \mathcal{T}_{x_k\rightarrow x_{k+1}}(u^{(k)}_{k-1})]$. 
   
   
   \label{Step1:LBFGS-update}
   \Else
      \State  
      Set $\tau_{k+1}=\tau_k$, $\{\rho_{k-l+1},\ldots,\rho_{k}\}=\{\rho_{k-l},\ldots,\rho_{k-1}\}$, 
      \Statex\qquad\qquad\enspace  $[s^{(k+1)}_{k-l+1}, \ldots, s^{(k+1)}_{k}]$ = 
       $[\mathcal{T}_{x_k\rightarrow x_{k+1}}(s^{(k)}_{k-l}), \ldots, \mathcal{T}_{x_k\rightarrow x_{k+1}}(s^{(k)}_{k-1})]$,
       \Statex\qquad\qquad\enspace  $[u^{(k+1)}_{k-l+1}, \ldots, u^{(k+1)}_{k}] = [\mathcal{T}_{x_k\rightarrow x_{k+1}}(u^{(k)}_{k-l}), \ldots, \mathcal{T}_{x_k\rightarrow x_{k+1}}(u^{(k)}_{k-1})]$. 
   \EndIf	
   \EndIf
    \EndFor
   \end{algorithmic}
   \end{algorithm}
   It has been shown in~\cite[Sec. 5]{huang2015broyden} that the Riemannian BFGS update
   (\ref{BFGSupdate}) can be equivalently written as $\mathcal{H}_k = \mathcal{V}_{k-1}^\flat\mathcal{\tilde{H}}_{k-1}\mathcal{V}_{k-1} + \rho_{k-1}s_{k-1}s_{k-1}^\flat$, where $\mathcal{V}_{k-1} = {\rm id}_{T_{x_{k-1}}\mathcal{M}}-\rho_{k-1}u_{k-1}s_{k-1}^\flat$, $\mathcal{V}_{k-1}^\flat = {\rm id}_{T_{x_{k-1}}\mathcal{M}}-\rho_{k-1}s_{k-1}u_{k-1}^\flat$, 
   $\mathcal{\tilde{H}}_{k-1}=\mathcal{T}_{x_{k-1}\rightarrow x_{k}}\circ\mathcal{H}_{k-1}\circ\mathcal{T}_{x_{k-1}\leftarrow x_{k}}$ and $\rho_{k-1} = \frac{1}{\langle u_{k-1},s_{k-1}\rangle}$. 
   When $l$ most recent $s$ and $u$, i.e., $(s_{k-1}, u_{k-1}), (s_{k-2}, u_{k-2}), \ldots (s_{k-l}, u_{k-l})$, are stored, the quasi-Newton operator $\mathcal{H}_k$ can be written as
      \begin{equation}\label{LBFGS}
     \begin{split}
         \mathcal{H}_{k} = &\mathcal{\tilde{V}}_{k-1}^\flat\mathcal{\tilde{V}}_{k-2}^\flat\cdots\mathcal{\tilde{V}}_{k-l}^\flat\mathcal{\tilde{H}}_{k}^0\mathcal{\tilde{V}}_{k-l}\cdots\mathcal{\tilde{V}}_{k-2}\mathcal{\tilde{V}}_{k-1} \\   &+\rho_{k-l}     \mathcal{\tilde{V}}_{k-1}^\flat\mathcal{\tilde{V}}_{k-2}^\flat\cdots\mathcal{\tilde{V}}_{k-l+1}^\flat s_{k-l}^{(k)}s_{k-l}^{(k)\flat}\mathcal{\tilde{V}}_{k-l+1}\cdots\mathcal{\tilde{V}}_{k-2}\mathcal{\tilde{V}}_{k-1}\\
         &+\cdots+ \rho_{k-1}s_{k-1}^{(k)}s_{k-1}^{(k)\flat}, \ \ k>1,
     \end{split}
   \end{equation}
   where 
   $\mathcal{\tilde{H}}_{k}^0 =\frac{\langle s_{k-1},u_{k-1}\rangle}{\langle u_{k-1},u_{k-1}\rangle} {\rm id}_{T_{x_k}\mathcal{M}}$ is the initial quasi-Newton operator for $k$, $s_i^{(k)}$ and $u_i^{(k)}$ represent tangent vectors in $T_{x_{k}}\mathcal{M}$ given by transporting $s_i$ and $u_i$, 
    $\tilde{\mathcal{V}}_{i} = {\rm id}_{T_{x_{k}}\mathcal{M}}-\rho_{i}u_{i}^{(k)}s_{i}^{(k)\flat}$ and $\tilde{\mathcal{V}}_{i}^\flat = {\rm id}_{T_{x_{k}}\mathcal{M}}-\rho_{i}s_{i}^{(k)}u_{i}^{(k)\flat}$. 
   If $m = k$, then it follows from Step \ref{dk2} of Algorithm~\ref{Algo:LRQNBM} that the search direction $d_k$ is computed by
   the two-loop recursion in~\cite[Algo. 2]{huang2015broyden} for limited-memory Riemannian BFGS. Otherwise, $d_k$ is computed by the limited-memory Riemannian SR1. Specifically,
   let $S_k$ and $U_k$ contain the $l$ 
   most recent $s_{k-1}$ and $u_{k-1}$, i.e.
   $$S_k=[s_{k-l}^{(k)},s_{k-l+1}^{(k)},\ldots,s_{k-1}^{(k)}]\ \ {\rm and} \ \ U_k=[u_{k-l}^{(k)},u_{k-l+1}^{(k)},\ldots,u_{k-1}^{(k)}].$$ The limited-memory Riemannian SR1 update is given by
   \begin{equation}\label{Lsr1}
       \mathcal{H}_k = \tau_k{\rm id}_{T_{x_k}\mathcal{M}} - (\tau_kU_k-S_k)(\tau_kU_k^{\flat}U_k-\tilde{R}_k-\tilde{R}_k^{\rm T}+C_k)^{-1}(\tau_kU_k-S_k)^{\flat},\ k>1,
   \end{equation}
   where we set $\tau_k = 1$ for every $k$, and $\tilde{R}_k$ is an upper triangular matrix given as
   \begin{equation*}  
   (\tilde{R}_k)_{ij}=\left\{  
        \begin{array}{ll}  
        \langle s_{k-l+i-1},u_{k-l+j-1}\rangle, & {\rm if}\ \  i\leq j; \\  
        0, &  {\rm otherwise},\\     
        \end{array}  
   \right.  
   \end{equation*} 
   $ (U_k^\flat U_k)_{ij} = \langle u_{k-l+i-1},u_{k-l+j-1}\rangle$, 
   and $C_k$ is a diagonal matrix of the form
   \begin{equation*}
       C_k = {\rm diag}\{\langle s_{k-l},u_{k-l}\rangle,\langle s_{k-l+1},u_{k-l+1}\rangle,\ldots,\langle s_{k-1},u_{k-1}\rangle\}.
   \end{equation*}
   Furthermore, regarding the action of $S_k^\flat$ and $U_k^\flat$, for any $\eta_k \in T_{x_k}\mathcal{M}$, we have
$$S_k^\flat \eta_k = (\langle s_{k-l}^{(k)}, \eta_k \rangle, \ldots,  {\color{magenta}\langle} s_{k-1}^{(k)}, \eta_k \rangle)^{\rm T}, \quad
U_k^\flat \eta_k = (\langle u_{k-l}^{(k)}, \eta_k \rangle, \ldots, \langle u_{k-1}^{(k)}, \eta_k \rangle)^{\rm T}.$$
The explicit computation for the action of the operator $\mathcal{H}_k$ on $\eta_k$ can be found in \cite[Sec. 3.2]{huang2022limited}.
   This expression is the inverse of the limited-memory SR1 update formula provided in \cite[Sec. 4]{huang2015riemannian}.
   
   \begin{remark}
       Algorithm \ref{Algo:LRQNBM} is obtained by replacing (\ref{SR1update}) and (\ref{BFGSupdate}) with (\ref{Lsr1}) and (\ref{LBFGS}) respectively, 
        and removing the correction steps \ref{Step:31}-\ref{Step:38} in Algorithm \ref{Algo:RQNBM}. The reason for this deletion of the correction is that if $\tilde{\rho} {\rm id}_{T_{x_k}\mathcal{M}}$ is added to $\mathcal{H}_k$ in Algorithm \ref{Algo:LRQNBM} (similar to Step \ref{Step: correction} in Algorithm \ref{Algo:RQNBM}), the correction does not work on $S_k$ and $U_k$ updates, i.e., $S_k$ and $U_k$ do not contain information of $\tilde{\rho} {\rm id}_{T_{x_k}\mathcal{M}}$. It follows that Steps \ref{Step:37}-\ref{Step:38} in  Algorithm~\ref{Algo:RQNBM} related to the correction are not applied.
        In addition, we observe that $\mathcal{H}_k$ cannot be represented by $\mathcal{H}_{k-1}$ in (\ref{Lsr1}) when the earliest vector pair $\{s_{k-l}^{(k)}, u_{k-l}^{(k)}\}$ needs to be deleted in $S_k$ and $U_k$. It leads to the conclusion that the boundedness of $\{w_k\}$ and $\{\mathcal{H}_k\}$ required in Theorem \ref{Thm:4.2} cannot be guaranteed. Therefore, the global convergence analysis in Section~\ref{sec4} cannot be used directly. We leave the improvement of Algorithm~\ref{Algo:LRQNBM} with global convergence guaranteed as a future work.
        %
   \end{remark}
   
   \section{Numerical experiment}\label{sec6}
   
   In this section, we investigate the performance of RQNBM (Algorithm \ref{Algo:RQNBM}) and M-RQNBM (Algorithm \ref{Algo:LRQNBM}), and compare them with two state-of-the-art methods including Riemannian proximal bundle method (RPBM) \cite{hoseini2021proximal} and Riemannian $\varepsilon$-subgradient method (RsubGM) \cite{grohs2016varepsilon}.
   The test problems, parameter setting and testing environment are given in Section~\ref{sec6.1}. The effectiveness of the quasi-Newton updates in Algorithms~\ref{Algo:RQNBM} and \ref{Algo:LRQNBM} is demonstrated in Section~\ref{Sec6.2}. In Section~\ref{subsec6.3}, it is shown that Algorithms~\ref{Algo:RQNBM}  and \ref{Algo:LRQNBM} are capable of solving problems with the size of thousands. Moreover, the computational times for solving the quadratic programming subproblems in the compared methods are investigated and reported. Finally, multiple nonsmooth optimization problems with various parameter settings are used to show the effectiveness and efficiency of the proposed methods.



   
   \subsection{Test problems, testing environment, and parameter setting}  \label{sec6.1}
   
   \paragraph{Test problems}
   The numerical experiments use four problems, including the maximum of multiple Rayleigh quotients, the sparse vector problem, the geometric median problem, and the bounding box problem.
   
   \textbf{Maximum of multiple Rayleigh quotients:}
   The maximum of multiple Rayleigh quotients problem is of the form
   \begin{align*}
     \min_{x \in S^n} \ \   & f(x)=\max_{i=1,2,\ldots,m}\frac{1}{2}x^{\rm T}A_ix, 
   \end{align*}
   where $S^n = \{x \in \mathbb{R}^{n+1} \mid x^{\rm T}x = 1\}$ denotes the unit sphere, and $A_i \in \mathbb{R}^{{(n+1)} \times {(n+1)}}, i=1,\ldots, m$ are given symmetric matrices. This problem has been used as a benchmark for testing nonsmooth optimization algorithms on manifolds; see e.g., \cite{grohs2016varepsilon,hoseini2021proximal}.
   The matrices $A_i, i = 1, \ldots, m,$ are selected as $A_i = {\rm diag}(1,2,\ldots, n+1) + 0.1\times (V_i + V_i^{\rm T} )$, where the elements of $V_i \in {\mathbb{R}}^{(n+1)\times (n+1)}$ are randomly selected using the MATLAB function {\bf randn} (Section \ref{Sec6.4}) or {\bf sprand} (Section \ref{subsec6.3}) with density 0.002. The latter setting of $A_i$ is used to simulate the performance of the tested algorithms for solving problems with computationally cheap cost functions.
   
   \textbf{The sparse vector problem:}
   The sparse vector problem \cite{7547961} finds the sparsest vector in an $(n+1)$-dimensional linear subspace $\mathcal{W}$ of $\mathbb{R}^m$. One formulation given in~\cite{7547961} is
   \begin{align*}
     \min_{x \in S^n} \ \   & \|Qx\|_1, 
   \end{align*}
   where $Q\in \mathbb{R}^{m\times (n+1)}$ denotes a matrix whose columns form an orthonormal basis of $\mathcal{W}$, and $\|v\|_1$ is defined by $\sum_{i = 1}^m |v_i|$.
   In this paper, we 
   do not force the matrix $Q$ to be orthonormal and 
       generate $Q$ by $Q = \begin{pmatrix}e_1 & Z \end{pmatrix}$,
   where $e_1 = (1, 0, \ldots, 0)^{\rm T}$ and $Z \in \mathbb{R}^{m \times n}$ is obtained by the function {\bf orth} in MATLAB. Therefore, the minimum value of the objective function is 1. 
   
   \textbf{The geometric median problem:}
   The geometric median on a Riemannian manifold considers the optimization problem
   \begin{align*}
     \min_{x \in \mathcal{M}} \ \   & \sum_{i=1}^{K}w_i\,{\rm dist}(x_i,x), 
   \end{align*}
   where $x_i \in \mathcal{M}$ are given data points, and $\{w_i\}_{i=1}^K$ denotes the weights satisfying $w_i > 0$ and $\sum_{i = 1}^K w_i = 1$. Such a problem has been considered in \cite{grohs2016varepsilon,hoseini2021proximal}. In this paper, the manifold $\mathcal{M}$ is chosen to be the unit sphere, i.e., $\mathcal{M} = S^n$. It follows that ${\rm dist}(x_i, x) = \arccos(x_i^{\rm T} x)$. Note that the objective function is nonsmooth only at $x = x_i$. Therefore, if the minimizer $x^*$ is not equal to any of $x_i$, then the objective function is smooth at $x^*$.
   In this paper, the weights are set by  $w_i=\frac{1}{K}, i = 1, \ldots, K$. The $K$ given points $x_1,x_2,\dots,x_K\in S^{n}$ are randomly generated by applying the Matlab function {\bf orth} to a vector, where each entry of the vector is drawn from the uniform distribution on $[0, 1]$.

   
   \textbf{The bounding box problem:}
       Finding a minimum volume box containing $K$ given points in $d$-dimensional space is called the oriented bounding box problem \cite{Borckmans2010OrientedBB}. Assuming that the matrix $E\in\mathbb{R}^{d\times K}$ is composed of the coordinates of the $K$ given points, the problem is formulated as
   \begin{align*}
     \min_{O \in \mathcal{O}_d} \ \   & f(O)=V(OE)=\prod_{i=1}^d (e_{i,\max}-e_{i,\min}),  
   \end{align*}
   where $\mathcal{O}_d = \{O \in \mathbb{R}^{d \times d} \mid O^{\rm T} O = I_d \}$ is the $d\times d$ orthogonal group, and $e_{i,\max}$ and $e_{i,\min}$ denote max and min elements of the $i$th row of $OE$, respectively. For a given $O$, the objective function $f$ is not differentiable at $O$ if multiple elements in any row reach the maximum or minimum. Such nondifferentiable points usually occur at minima; see \cite{Borckmans2010OrientedBB}.
   In this paper, each entry of $E \in \mathbb{R}^{d\times K}$ is drawn from the uniform distribution  on $[0, 0.75]$.

   \paragraph{Testing Environment}
   All codes are written in MATLAB R2020a and run on a PC Intel Core I5 with CPU 2.4 GHZ and 16GB of RAM. The codes for reproducing the experimental results are available at 
   \url{https://www.math.fsu.edu/~whuang2/papers/RMQNBM.htm}.

   \paragraph{Parameter setting}
   The Riemannian metrics of $S^n$ and $\mathcal{O}_d$ are endowed from the Euclidean space, i.e., $\langle \eta_x,\xi_x\rangle_{x} = \eta_x^{\rm T} \xi_x$ for any $x\in S^{n}, \eta_x, \xi_x \in T_x S^n$ and $\langle \eta_x,\xi_x\rangle_{x} = \mathrm{trace}(\eta_x^{\rm T} \xi_x)$ for any $x\in \mathcal{O}_d, \eta_x, \xi_x \in T_x \mathcal{O}_d$. The retraction on $S^n$ and $\mathcal{O}_d$ is chosen to be $R_x(\eta_x) = {\rm qf}(x + \eta_x)$, where ${\rm qf}(M)$ denotes the ${\rm Q}$ factor of the ${\rm QR}$ decomposition with the diagonal elements of factor ${\rm R}$ being positive. The vector transport by parallelization in~\cite[Sec. 4]{huang2017intrinsic}, which is isometric, is used. Moreover, the locking condition is further guaranteed by using the modification in~\cite[Sec. 4.1]{huang2015broyden} to the vector transport by parallelization.
     The vector transport $\mathcal{T}$ selected in this study satisfies Assumption \ref{assum3}. At the same time, the sphere $S^{n}$, the Stiefel manifold ${\rm St}(n,p)$, and the orthogonal group $\mathcal{O}(d)$ involved in the four problems are all compact Riemannian manifolds. This key property ensures that they naturally satisfy Assumptions \ref{assum4} and \ref{assum1}.
    
   The parameters of Algorithm \ref{Algo:RQNBM} and Algorithm \ref{Algo:LRQNBM} are set to be
   $t_{\rm min} = 2.22\times10^{-16}$, $t_{\max}=1$, $\mu_0=0.18$, $D=1$, ${\theta_{\rm A}}={\theta_{\rm L}}=0.1$, ${\theta_{\rm R}}=0.45$, $\theta_{\rm T}=0.2$, $\gamma=0.15$, $\theta = 1$, $\kappa=0.25$, $\nu =2$, $\tilde{D} = 10^{4}$, and $\Gamma=50$. For all values $k<2^d\times n$ (if $d$ does not exist in an example, we set $d=0$), we set $\rho=0.1$. For $k \geq 2^d\times n$, we set $\rho=10^{-3}$. 
   The parameter $\tilde{\rho} = 0.1$ for $k<500$, otherwise $\tilde{\rho} = 10^{-3}$.  Unless otherwise indicated, the parameter $m_l$ is set to $8$. 
   In Step \ref{LS:14} of Procedure \ref{Procedure:LS}, the point $t$ located at one-quarter of the interval length from the left endpoint is selected. The initial {\color{magenta}symmetric positive-semidefinite} matrix is set to be the identity. The initial point is generated by applying the MATLAB function {\bf{orth}} to a vector (matrix), where the vector (matrix) is obtained by the function {\bf{rand}} ({\bf{randn}}) in MATLAB. 
   The default parameters in the two compared methods (RsubGM and RPBM) are used except that in RsubGM the Armijo parameter $c$ is set to be $5\times 10^{-3}$ and in RPBM the injectivity radius $\varepsilon$ and the descent parameter $m_L$ are set to be $0.05$ and $2\times 10^{-4}$ respectively. These changes are made so that the compared algorithms perform better when solving large-scale problems.
   

   Unless otherwise indicated, all the tested algorithms terminate if the norm of the corresponding search direction is smaller than $10^{-5}$ or the number of iterations exceeds 5000. The quadratic programming problem in all the tested algorithms is solved by the {\bf quadprog} function in MATLAB. When an average of random runs is reported, the problem generations
   and the initial points of all the tested algorithms are the same.

   \subsection{Performance for updating $\mathcal{H}_k$}\label{Sec6.2}
   
   Table~\ref{table1} contains an average result of 10 runs using different dimensions for solving the oriented bounding box problem. 
   The parameter $K$ is set to be 1000 and multiple values of $d$ are used, i.e., $d=3,4,5,6,8,10$. RQNBM-NO denotes Algorithm \ref{Algo:RQNBM} with the updates of $\mathcal{H}_k$ being skipped. We say that all algorithms find the same minimum if the local minimizers returned by any two algorithms satisfy $\| \hat{x}_* - \tilde{x}_*\| \leq 10^{-3}$, where $\hat{x}_*$ and $\tilde{x}_*$ are outputs from the two algorithms. Note that the bounding box problem may have multiple local minimizers. For a fair comparison, an average of 10 runs is taken where all the algorithms converge to the same minimizer\footnote{If the algorithms converge to different minimizers, then the convergence speeds of those algorithms may differ due to the various landscapes of the cost function around the different minimizers, not due to the differences between the algorithms.}.
   
   
    
   Table \ref{table1} shows that the RQNBM and M-RQNBM with appropriate memory size $m_l$ outperform RQNBM-NO in the sense of the number of function evaluations and the computational time.
   This coincides with the intuition that quasi-Newton update improves performance.
   Moreover, the computational time and the number of function evaluations decrease as $m_l$ increases in M-RQNBM. This implies that a larger memory size in M-RQNBM gives a better quasi-Newton update operator and reduces the function evaluations. However, a larger memory size increases the computational cost when computing $d_k$. For the bounding box problems, the dominant computation is on the function evaluation. It follows that the total computational time decreases as $m_l$ increases. Such behaviors may not occur for other problems.
   
   Though the computational time of solving quadratic programming in RQNBM and RQNBM-NO accounts for a comparatively large proportion of the total time in this section, this is not typical. Note that the proposed methods solve $3$-dimensional quadratic programming problems in null steps, which is independent of the dimension of the domain manifold. Therefore, as the dimension of the manifold increases, 
   the proportion of the computational time for quadratic programming is expected to decrease in general, as shown in Section~\ref{subsec6.3} where large-scale problems are tested.

   \begin{table}[!htbp]
   \caption{
   An average result of 10 random runs using multiple dimensions oriented bounding box problems. The subscript $k$ indicates a scale of $10^k$.
   $f_{\rm opt}$, $n_f$, $t_{\rm quad}$, and $t$ respectively denote the final objective value, the number of function evaluations, the time of quadratic programming calculation, and the run time. \label{table1} }
   \renewcommand{\arraystretch}{1.2}
   \setlength{\tabcolsep}{15pt}
   \begin{tabular}{l |c |c c c c c}
   \hline
   \multicolumn{1}{l|}{\multirow{1}*{Method}  } &\multicolumn{1}{c|}{\multirow{1}*{$m_l$}}   &\multicolumn{1}{c}{\multirow{1}*{$f_{\rm opt}$}} &\multicolumn{1}{c}{\multirow{1}*{$n_{f}$}  }  &{\multirow{1}*{$t_{\rm quad}$}}   &\multicolumn{1}{c}{\multirow{1}*{$t$}}\\
   
   \hline  
  RQNBM-NO  &$-$& $2.15_{-1}$  & $1.17_{4}$    & $2.55$  & $5.45$   \\
   RQNBM   &$-$&  $2.15_{-1}$  & $2.61_{3}$   & $8.41_{-1}$  & $1.58$  \\
   M-RQNBM  &4& $2.15_{-1}$  & $7.86_{3}$   & $1.73$  & $5.55$ \\
   &8& $2.15_{-1}$  & $4.90_{3}$    & $1.10$  & $4.46$   \\
   &16& $2.15_{-1}$  & $2.05_{3}$   & $4.58_{-1}$  & $2.67$   \\
   &32& $2.15_{-1}$  & $6.83_{2}$   & $1.46_{-1}$  & $1.19$   \\
   \hline
   \end{tabular}
   \end{table}

   \subsection{Test large-scale problems} \label{subsec6.3}
   
   The maximum of multiple Rayleigh quotients is chosen as the representative problem for comparing the proposed methods with the two state-of-the-art methods, RPBM and RsubGM.
   We say that an optimal solution is found if the value of objective function $f$ satisfies
   \begin{equation*}
       0\leq \frac{f-f_{\rm opt}}{\vert f_{\rm opt}\vert +1}\leq 10^{-4},
   \end{equation*} 
   where $f_{\rm opt}$ is the known optimal objective function value or the minimum value found by the used algorithms.
   
   The results of RQNBM, M-RQNBM, RPBM, and RsubGM with $n =5000,10000$, $m=2$ are reported in Table \ref{table2}.  The bundle size is set to be $n/100$ in RPBM. Since the paper of RPBM~\cite[Sec. 5]{hoseini2021proximal} proposes to use the exponential mapping and qf retraction, 
   RQNBM and M-RQNBM are compared to both RPBM-exp and RPBM-qf, where the postfixes ``exp'' and ``qf'' respectively denote the combination with the exponential mapping and the qf retraction.
   As shown in Table \ref{table2}, RsubGM and RPBM take more numbers of function evaluations and spend more time solving the quadratic programming subproblems. This is consistent with the intuitions that (i) quasi-Newton updates in RQNBM and M-RQNBM improve the performances in the sense of reducing the function evaluation, and (ii) the dimension of the quadratic programming subproblems in RPBM and RsubGM may increase as the dimension of the manifold increases whereas the dimension of the subproblem in RQNBM is always three. Note that for large-scale problems, RQNBM may take much computational time in updating the quasi-Newton operator. Therefore, the most efficient algorithm is the M-RQNBM under the circumstances of this section.

   \begin{table}[!htbp]
   \caption{An average result of 10 random runs using different dimensions for the maximum of multiple Rayleigh quotients. $t_f$ denotes the computational time of the function evaluations.\label{table2} }
   \renewcommand{\arraystretch}{1.5}
   \setlength{\tabcolsep}{4pt}
   \begin{tabular}{l| c c c c c |c c c c c}
   \hline
   
   \multicolumn{1}{c|}{\multirow{1}*{$n$}  }  &\multicolumn{5}{c}{\multirow{1}*{$5000$}  } &\multicolumn{5}{|c}{\multirow{1}*{$10000$}}\\
   \hline
   \multicolumn{1}{c|}{} &$f_{\rm opt}$ &$n_{f}$ &$t_f$  &$t_{\rm quad}$ &$t$ &$f_{\rm opt}$ &$n_{f}$ &$t_f$  &$t_{\rm quad}$ &$t$\\
   \hline
   
   \hline
    
   RQNBM &  $5.00_{-1}$  & $1.58_{3}$  & $2.80_{-1}$ &  $1.04$  &$5.25_{1}$ & $5.00_{-1}$  & $2.68_{3}$  & $1.61$ &  $1.91$  &$3.87_{2}$ \\
   M-RQNBM &  $5.00_{-1}$  & $2.87_{3}$  & $4.31_{-1}$ &  $9.73_{-1}$  &$6.99$ & $5.00_{-1}$  & $5.11_{3}$  & $2.87$ &  $2.68$  &$2.75_{1}$ \\
   RPBM-exp&  $5.00_{-1}$  & $1.28_{4}$   & $1.88$  & $2.73_{1}$ & $4.84_{1}$ &  $5.00_{-1}$  & $2.47_{4}$   & $1.34_{1}$  & $3.21_{2}$  &$4.53_{2}$ \\
   RPBM-qf  & $5.00_{-1}$  & $1.20_{4}$    & $1.85$  & $2.57_{1}$  &$4.56_{1}$ & $5.00_{-1}$  & $2.54_{4}$   & $1.40_{1}$  & $3.31_{2}$  &$4.57_{2}$ \\
   RsubGM & $5.00_{-1}$  & $9.22_{3}$   & $1.57$  & $7.02_{1}$ & $8.56_{1}$& 
   $5.00_{-1}$   & $1.34_{4}$  & $7.46$  & $4.23_{2}$  & $5.24_{2}$  \\
   \hline
   
   \end{tabular}
   \end{table}
   
   \subsection{Comparison of RQNBM, M-RQNBM, RPBM, and RsubGM }\label{Sec6.4}
   
   The performance profiles in~\cite{dolan2002benchmarking} are used to compare the performance of different algorithms. Specifically, define $t_{p, s}$ to be the computational time required to solve problem $p$ by an algorithm $s$ and define the performance ratio to be
   $
   r_{p, s} = \frac{t_{p, s}}{ \min\{ t_{p, s} : s \in \mathcal{S} \} },
   $
   where $\mathcal{S}$ denotes the set of algorithms. The probability for an algorithm $s \in \mathcal{S}$ to have a performance ratio $r_{p, s}$ smaller than a factor $\tau$ is defined by
   $
   \rho_s(\tau) = \frac{1}{\sharp(\mathcal{P})} \sharp \{ p \in \mathcal{P} : r_{p, s} \leq \tau \},
   $
  where $\sharp$ denotes the number of entries in the following set, and $\mathcal{P}$ denotes the set of tested problems. 
   The performance profiles of the tested algorithms are shown by plotting $\rho_s$, for all $s \in \mathcal{S}$.

   Since the performance of the tested algorithms varies based on the size of the problems, we split the tested problems as shown in Table~\ref{table3}\footnote{Note that ``small'', ``medium'', and ``large'' refer to {\bf RELATIVELY} small, medium, and large size problems.}. For each combination of parameters, 10 random runs are used. The performance profiles of RQNBM, RPBM-exp, RPBM-qf, RsubGM, and M-RQNBM for small-, medium-, and large-scale problems are reported respectively in Figs.~\ref{fig.1}, \ref{fig.2}, and \ref{fig.3}.
   It can be seen that the proposed methods, RQNBM and M-RQNBM, outperform RPBM-exp, RPBM-qf, and RsubGM since the curves of RQNBM and M-RQNBM are on the upper-left corner of the figures. In Figs. \ref{fig.1} and \ref{fig.2}, the computational time of RQNBM is smaller than or comparable to 
   M-RQNBM, whereas in Fig.~\ref{fig.3}, M-RQNBM is faster than RQNBM. Such performance verifies that M-RQNBM reduces the computational time of the quasi-Newton updates.

   \begin{table}[!htbp]
   \caption{The parameter settings for small/medium/large scale problems. MRQ, SVP, GMP, and BBP respectively denote the maximum of multiple Rayleigh quotients, the sparse vector problem, the geometric median problem, and the bounding box problem. \label{table3} }
   \renewcommand{\arraystretch}{1.5}
   \setlength{\tabcolsep}{4pt}
   \begin{tabular}{c| c c c c }
   \hline 
   Parameters &  MRQ & SVP & GMP & BBP \\
   \hline 
   \multirow{2}*{Small} & $n=5$ & $n=3,6,9$ & $n=3,5,10,20,50$  & $d=3,4$  \\
    &  $m=20,100,200,500$ & $m=10n$ &  $K=5000$ & $K=1000$ \\
   \hline 
   \multirow{2}*{Medium} & $n=10$  & $n=12, 15$  & $n=100,200,500$  & $d=5,6$ \\
    & $m=20,100,200,500$  & $m=10n$  & $K=5000$  & $K=1000$ \\
   \hline 
   \multirow{2}*{Large} & $n=20$  & $n=18,21$  &  $n=1000,2000$ & $d=8, 10$ \\
   & $m=20,100,200,500$ & $m=10n$  &  $K=5000$ & $K=1000$ \\
   \hline
   \end{tabular}
   \end{table}

   \captionsetup[figure]{labelfont={bf},name={Fig.},labelsep=space}

   \begin{figure}[!htbp]
   \centering
   \begin{minipage}[t]{1\textwidth}
   \centering
   \includegraphics[width=10cm]{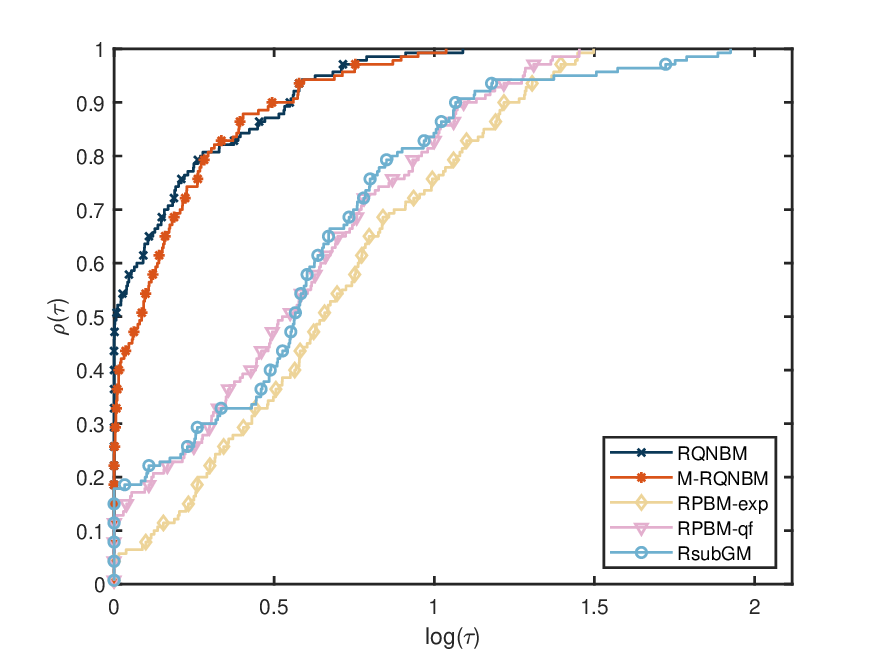}
   \end{minipage}
   \caption{The performance profiles for the small-scale problems.}
   \label{fig.1}
   \end{figure}

   \begin{figure}[!htbp]
   \centering
   \begin{minipage}[t]{1\textwidth}
   \centering
   \includegraphics[width=10cm]{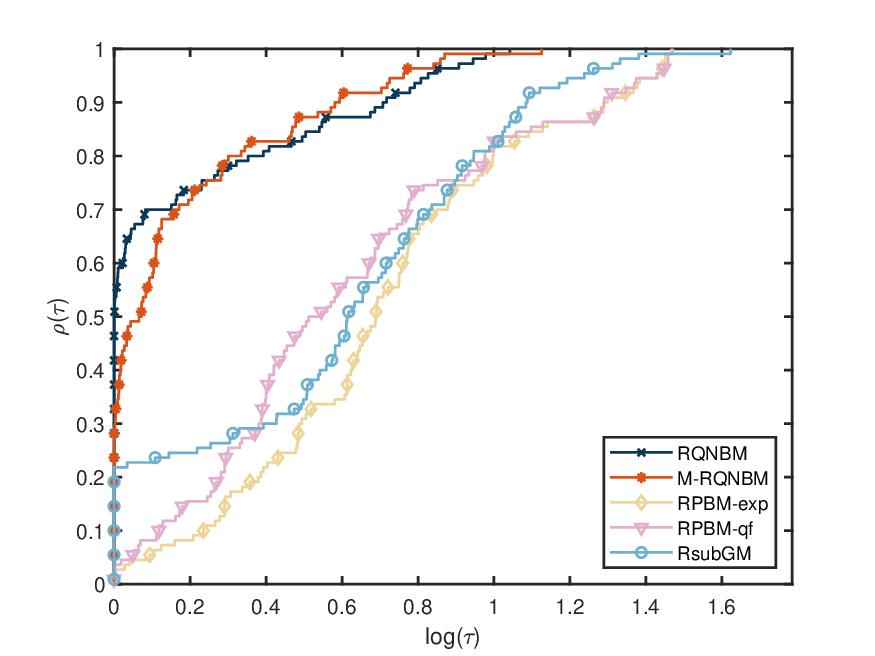}
   \end{minipage}
   \caption{The performance profiles for the medium-scale problems.}
   \label{fig.2}
   \end{figure}
   

   \begin{figure}[!htbp]
   \centering
   \begin{minipage}[t]{1\textwidth}
   \centering
   \includegraphics[width=10cm]{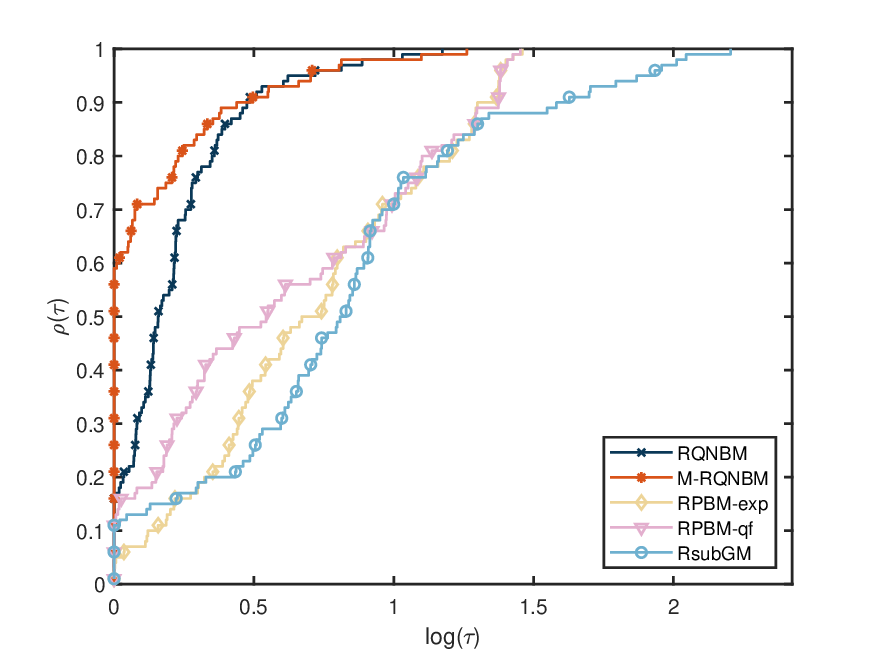}
   \end{minipage}
   \caption{The performance profiles for the large-scale problems.}
   \label{fig.3}
   \end{figure}

   \section{Conclusions}\label{sec7}
   In this paper, we have proposed a restricted memory quasi-Newton bundle method for minimizing locally Lipschitz {\color{magenta}continuous} functions over Riemannian manifolds. The method approximates the curvature information of the objective function using Riemannian versions of the quasi-Newton updating formulas. The subgradient aggregation technique is introduced to avoid solving time-consuming quadratic subproblems, and a new Riemannian line-search procedure is proposed to generate stepsizes. Global convergence is proved, and a modified algorithm with limited-memory quasi-Newton updates is presented.
   Numerical experiments demonstrate that the proposed methods have robust performance compared with the closely related methods. 
   


   \subsection*{Acknowledgments}
   The authors sincerely thank the Editor-in-Chief and the two anonymous reviewers for their valuable comments. Their professional suggestions on mathematical rigor, notation consistency and presentation details have greatly improved the quality and readability of this paper.
   
   \subsection*{Statements and Declarations}
   \subsubsection*{Funding information}
   This work was supported by the National Natural Science Foundation of China (Nos.  12271113, 12371311, 12171106), the Natural Science Foundation of Fujian Province (No. 2023J06004), the Fundamental Research Funds for the Central Universities (No. 20720240151) and Xiaomi Young Talents Program.
   \subsubsection*{Conflict of interest}
    The authors declare that they have no conflict of interest.
   \subsubsection*{Availability of data and materials} 
   The authors declare that the codes and data supporting the findings of this study are available within the paper.

\bibliography{RQNBM1}

\end{document}